\documentclass[11pt,a4paper]{amsart}
\usepackage{color}
\usepackage{amssymb,amsmath,amsthm,amstext,amsfonts}
\usepackage[dvips]{graphicx}
\usepackage{psfrag}
\usepackage{url}
\usepackage{amsfonts}
\usepackage{amsmath}
\usepackage{mathrsfs}
\usepackage{graphicx}
\usepackage{amsmath,amsthm,amssymb,amscd}
\usepackage{color}
\usepackage{enumerate}
\usepackage{tikz}
\usepackage{pgf}
\usepackage[colorlinks,linkcolor=black,anchorcolor=black,citecolor=black,hyperindex=true,CJKbookmarks=true]{hyperref}
\usetikzlibrary{arrows, decorations.pathmorphing, backgrounds, positioning, fit, petri, automata}

\usepackage{epsfig}

\makeatletter \@addtoreset{equation}{section} \makeatother

\renewcommand\thetable{\thesection.\@arabic\c@table}

\theoremstyle{plain}
\newtheorem{maintheorem}{Theorem}

\newtheorem{mainlemma}{Lemma}

\newtheorem{theorem}{Theorem}[section]

\newtheorem{lemma}{Lemma}[section]

\newtheorem{definition}{Definition}[section]

\newtheorem{Thm}{Theorem}[section]
\newtheorem{Lem}[Thm]{Lemma}
\newtheorem{Prop}[Thm]{Proposition}
\newtheorem{Cor}[Thm]{Corollary}
\theoremstyle{remark}
\newtheorem{Def}[Thm] {Definition}
\newtheorem{Rem}[Thm] {Remark}

\newcommand{\htop}{h_{\topp}}

\newcommand{\eps}{\varepsilon}

\newcommand{\N}{{\mathbb{N}}}
\newcommand{\set}[1]{\left\{#1\right\}}

\long\def\begcom#1\endcom{}

\newcommand{\orb}{\operatorname{orb}}
\newcommand{\length}{\operatorname{\length}}

\def\length{\operatorname{length}}

\newcommand{\bl} {\begin{lemma}}
\newcommand{\el} {\end{lemma}}
\newcommand{\bt} {\begin{theorem}}
\newcommand{\et} {\end{theorem}}

\newcommand{\bp}{\begin{proof}}
\newcommand{\ep}{\end{proof}}
\newcommand  {\ee} {\end{equation}}
\newcommand  {\beq} {\begin{eqnarray*}}
\newcommand  {\eeq} {\end{eqnarray*}}

\newcommand  {\bd} {\begin{definition}}
\newcommand  {\ed} {\end{definition}}

\newcommand{\diam}{\operatorname{diam}}

\newcommand{\topp}{\operatorname{top}}

\def\ep{\noindent{\hfill $\Box$}}

\begin{document}

\title{Strongly Distributional Chaos in the Sets of Twelve Different Types of Non-recurrent Points }

\author{An Chen, Xiaobo Hou, Wanshan Lin and Xueting Tian}
\address{An Chen, School of Mathematical Sciences,  Fudan University\\Shanghai 200433, People's Republic of China}
\email{15210180001@fudan.edu.cn}

\address{Xiaobo Hou, School of Mathematical Sciences,  Fudan University\\Shanghai 200433, People's Republic of China}
\email{20110180003@fudan.edu.cn}

\address{Wanshan Lin, School of Mathematical Sciences,  Fudan University\\Shanghai 200433, People's Republic of China}
\email{21110180014@m.fudan.edu.cn}

\address{Xueting Tian, School of Mathematical Sciences,  Fudan University\\Shanghai 200433, People's Republic of China}
\email{xuetingtian@fudan.edu.cn}


\begin{abstract}

   In present paper we mainly focus on non-recurrent dynamical orbits with empty syndetic center and show that twelve  different statistical structures over mixing  expanding maps or transitive Anosov  diffeomorphisms all have dynamical complexity in the sense of strongly distributional chaos.

\end{abstract}



\maketitle
\section{Introduction}


The study of the thermodynamic formalism and multifractal
analysis for maps with some hyperbolicity has drawn the attention of many researchers
from the theoretical physics and mathematics communities in the last
decades. The general concept of multifractal analysis (or dimension theory), that can be traced back to
Besicovitch, is to decompose the phase space in (invariant) subsets of points which have a
similar dynamical behavior and to describe the size of each of such subsets from
the geometrical or topological viewpoint by using the concepts of Hausdorff dimension, topological entropy or pressure and Lebesgue measure etc.

Birkhoff ergodic theorem is one classical and basic way to study dynamical orbits by describing asymptotic behavior from the probabilistic viewpoint of a given observable function. Some concepts in the theory of multifractal analysis, for example, irregular set and level set, derive from the Birkhoff ergodic theorem. As for characterizing the complexity of these sets, there are some prevalent indexes. For example, Pesin and Pitskel \cite{Pesin-Pitskel1984} are the first to notice the phenomenon of the irregular set carrying full topological entropy in the case of the full shift on two symbols. Since then, there are lots of advanced results to show that the irregular points can carry full entropy(and topological pressure) in symbolic systems, hyperbolic systems and systems with specification-like or shadowing-like properties \cite{Barreira-Schmeling2000,BPS,  Olsen2002, Olsen2003,Olsen-Winter,  Pesin1997,  FFW,CKS,TDbeta,DOT,Thompson2009,Thompson2008}. In addition, Lebesgue measure \cite{Takens,KS}, Hausdorff dimension \cite{Pesin-Pitskel1984,Barreira-Schmeling2000,  Pesin1997, CKS,TDbeta,DOT,Bar2011,TV vp,FFW} and chaos \cite{CT} are all used afterwards to describe how complex the irregular set and level set can be. We refer to \cite{Fal,Pesin1997,BaL} for more contents of fractals and dimension theory.

Another way to differ asymptotic behavior of dynamical orbits is from the perspective of periodic-like recurrence. There are many  such concepts, for example, periodic points, almost periodic points, weakly almost periodic points,  transitive points etc \cite{Sig,Kal,Katok,Gottschalk46,Grillenberger,Gottschalk44-2222,Gottschalk44,Rees,Zhou93,T16}.
For periodic points, it is well-known that  the exponential growth of the periodic points equals to topological entropy for hyperbolic systems but Kaloshin showed that in general periodic points can grow much faster than entropy \cite{Kal}. Moreover,
it is well known that for $C^1$ generic diffeomorphisms, all periodic points are hyperbolic so that countable and they form a
dense subset of the non-wandering set (by  classical  Kupka-Smale theorem, Pugh's  or Ma\~{n}$\acute{\text{e}}$'s
ergodic closing lemma from   smooth ergodic theory, for example,  see  \cite{Kupka1,Smale1,Pugh1,Pugh2,Mane1}).
However, periodic point does not exist naturally. For example there is no periodic points in any irrational rotation.
Almost periodic point is a good generalization which exists naturally  since it is equivalent that it belongs to a minimal set (see \cite{Birkhoff,Gottschalk44-2222,Gottschalk44,Gottschalk46,Mai}) and by Zorn's lemma  any dynamical system contains at least one
minimal invariant subset.  There are many examples of subshifts which are strictly ergodic (so that every point in the subshift is almost periodic)
and has positive entropy, for example, see \cite{Grillenberger}. In other words, almost periodic points have strong dynamical complexity.
For weakly almost periodic points, quasi-weakly almost periodic points and Banach-recurrent points, it is shown recently that all their gap-sets with same asymptotic behavior carry high dynamical complexity in the sense of full topological entropy over dynamical system with specification-like property and expansiveness \cite{HTW,T16}. In addition to studying the entropy of periodic-like recurrent point sets, authors in \cite{CT} consider their complexity from the perspective of chaos. The results of \cite{HTW,T16,CT} are all restricted on studying transitive points (or dense orbits).

In contract to transitive points, non-dense points (i.e., points that are not transitive)
have been studied a lot. For example, non-dense points (i.e., points that are not transitive)
form a set with full Hausdorff dimesion for hyperbolic systems, see \cite{Urbanski non-dense}. An effective
tool to show full Hausdorff dimension is Schmidt's game, which was first introduced by
Schmidt in \cite{Schmidt1966}. A winning set for such games is large in the following sense: it is dense in
the metric space, and its intersection with any nonempty open subset has full Hausdorff
dimension when the metric space is $R^n$ or a manifold. In \cite{Schmidt1966} Schmidt showed that the set
of badly approximable numbers is winning for Schmidt's game and hence has full Hausdorff
dimension 1. There are many applications of Schmidt's game in homogeneous dynamics
due to the well known connection between Diophantine approximation and bounded orbits
of flows (for example, see \cite{Dani1985,Dani1986,Dol-Anosov,Kleinbock-Weiss2010,Wu2016-schmidt} and references therein). Recently it was
showed in \cite{DT} that for hyperbolic or expanding systems the set of non-dense points
carries full topological entropy by using a way different with Schmidt's game. Moreover,
this set may contain eighteen different types for which twelve fractals correspond to nonrecurrent
points and other six fractals correspond to recurrent (but not transitive) points, and every one of eighteen fractals has full topological entropy even though the twelve different fractals corresponding to non-recurrent points always have totally zero measure
(since the set of recurrent points has totally full measure). In present paper, we adopt this
new pattern of classification of non-recurrent points and show that twelve different fractals
of non-recurrent points all are strongly distributional chaoic.

\subsection{Non-recurrent Points and Statistical $\omega$-limit Sets}
Let $(X,d)$ be a nondegenerate $($i.e,
with at least two points$)$ compact metric space, and $f:X \rightarrow X$ be a continuous map. Such $(X,f)$ is called a dynamical system. 
Given a dynamical system $(X,f).$ Let $\mathfrak{C}=\langle x_n\rangle_{n=1}^{+\infty}$, we define $\omega(\mathfrak{C}):=\bigcap_{n\geq 1}\overline{\bigcup_{k\geq n}\{x_k\}}$. For any $x \in X$, the orbit of $x$ is the sequence $\langle x, fx, f^2x \cdots\rangle$, which we denote by $orb(x,f)$. We call $\omega(orb(x,f))$ the $\omega$-$limit$ set of $x$, written as $\omega_f(x)$ briefly. A point $x \in X$ is $recurrent$, if $x \in \omega_f(x)$. We denote all recurrent points of $X$ by $Rec$ or $Rec(f)$. We call $X\setminus Rec$ the non-recurrent point set, which we denote by $NR$ or $NR(f)$. A point $x$ is transitive if $\omega_f(x)=X$. It is equivalent to $orb(x, f) = X$ if the system is surjective. Let $Tran, ND$ denote the set of transitive points and non-transitive points. Note that
$$ND=NR\sqcup(Rec\setminus Tran).$$

If for every pair of nonempty open sets $U,$ $V$ there is an positive integer $n$ such that $f^n(U)\cap V\neq \emptyset$ then we call $(X,  f)$ \textit{topologically transitive}.
Furthermore,   if for every pair of nonempty open sets $U,$ $V$ there exists an positive integer $N$ such that $f^n(U)\cap V\neq \emptyset$ for every $n>N$,   then we call $(X,  f)$ \textit{topologically mixing}. 
\begin{maintheorem}\label{thm-NR-new}
Suppose $(X,  f)$ is a mixing expanding map or a transitive Anosov diffeomorphism on a compact  manifold. Then the set of non-recurrent set $NR(f)$ is strongly distributional chaotic. 
\end{maintheorem}
Strongly distributional chaos is a kind of chaos which is stronger than usual distributional chaos and Li-Yorke chaos. See their definitions in section \ref{subsection-chaos}. 

Since $NR(f)$ is contained in the set of non-dense points $ND(f)$, then one also has that $ND$  is strongly distributional chaotic. We will also show that $Rec(f)\setminus Tran(f)$ is distributional chaoic of type 1 in section \ref{recnotdense}.
Since $NR(f)$ has zero measure for any invariant measure so that we have a following
corollary.

\begin{Cor}
Suppose $(X,  f)$ is a mixing expanding map or a transitive Anosov diffeomorphism on a compact manifold.  Then there is set $S\subseteq X$ which is strongly distributional chaotic  such that for any $\mu\in\mathcal M_f(X)$, $\mu(S)=0$.
\end{Cor}

The set of non-recurrent points $NR$ also can be decomposed into many  different levels by using different asymptotic behavior. One natural question is that
$$\text{\it Which layer does strongly distributional chaos occur in?}$$  Here we use some concepts from \cite{DT} 
 to describe different statistical structure of dynamical orbits.

\begin{Def}
Let $S\subseteq \mathbb{N}$,   define
$$\bar{d} (S):=\limsup_{n\rightarrow\infty}\frac{|S\cap \{0,  1,  \cdots,  n-1\}|}n,  \,  \,   \underline{d} (S):=\liminf_{n\rightarrow\infty}\frac{|S\cap \{0,  1,  \cdots,  n-1\}|}n,  $$
where $|Y|$ denotes the cardinality of the set $Y$.   These two concepts are called {\it upper density} and {\it lower density} of $S$,   respectively.   If $\bar{d} (S)=\underline{d}(S)=d,  $ we call $S$ to have density of $d.  $  Define $$B^* (S):=\limsup_{|I|\rightarrow\infty}\frac{|S\cap I|}{|I|},  \,  \,   B_* (S):=\liminf_{|I|\rightarrow\infty}\frac{|S\cap I|}{|I|},  $$
here $I\subseteq \mathbb{N}$ is taken from finite continuous integer intervals.   These two concepts are called {\it Banach upper density} and {\it Banach lower density} of $S$,   respectively.
 A set $S\subseteq \mathbb{N}$ is called {\it syndetic},   if there is $N\in\N$ such that for any $n\in\N,  $ $ S\cap \{n,  n+1,  \cdots,  n+N\}\neq \emptyset.    $
  \end{Def}These concepts of density are basic and have played important roles in the field of dynamical systems,   ergodic theory and number theory,   etc.
Let $U,  V\subseteq X$ be two nonempty open subsets and $x\in X.  $
 Define
  sets of visiting time
 $$N (U,  V):=\{n\geq 1\mid  U\cap f^{-n} (V)\neq \emptyset\} \,  \,  \text{ and } \,  \,  N (x,  U):=\{n\ge 1\mid  f^n (x)\in U\}.  $$

\begin{Def} (Statistical $\omega$-limit sets)
 For $x\in X$ and $\xi=\overline{d},   \,  \underline{d},   \,  B^*,  \,   B_*$,   a point $y\in X$ is called $x$-$\xi$-accessible,   if for any $ \varepsilon>0,  \,  N (x,  V_\varepsilon (y))\text{ has positive   density w.  r.  t.   }\xi,  $ where     $V_\varepsilon (x)$ denotes  the  ball centered at $x$ with radius $\varepsilon$.
 Let $$\omega_{\xi}(x):=\{y\in X\,  |\,   y\text{ is } x-\xi-\text{accessible}\}.  $$  For convenience,   it is called {\it $\xi$-$\omega$-limit set of $x$}.
We also call $\omega_{B_*}(x)$ to be {\it syndetic center} of $x$.
 \end{Def}
  It is obvious that
\begin{equation}\label{omega-relation}
  \omega_{B_*}(x)\subseteq \omega_{\underline{d}}(x)\subseteq \omega_{\overline{d}}(x)\subseteq \omega_{B^*}(x)\subseteq \omega_f(x).
\end{equation}
In \cite{AAN} for maps and \cite{AF} for flows  that $\omega_{\overline{d}}(x)$ is called essential $\omega$-limit set of $x$.
  In \cite{Zhou93,DT}, authors show that these concept have a close relationship with measure space.
 From \cite{DT}, we know that
for any $x\in X$, if $\omega_{B_*}(x)=\emptyset,$ then $x$ satisfies only one of following twelve cases:
  \begin{description}

  \item[Case  (1)   ]  \,  \  $\emptyset=\omega_{B_*}(x)\subsetneq    \omega_{\underline{d}}(x)=
\omega_{\overline{d}}(x)  \subsetneq \omega_{B^*}(x)= \omega_f(x) ;$
\item[Case  (1')   ]  \,  \   $\emptyset=\omega_{B_*}(x)\subsetneq  \omega_{\underline{d}}(x)=
\omega_{\overline{d}}(x)  \subsetneq \omega_{B^*}(x)\subsetneq \omega_f(x) ;$

  \item[Case  (2)    ]  \,  \   $    \emptyset=\omega_{B_*}(x)=\omega_{\underline{d}}(x)\subsetneq
\omega_{\overline{d}}(x)= \omega_{B^*}(x)= \omega_f(x) ;$

\item[Case  (2')    ]  \,  \  $\emptyset=\omega_{B_*}(x)=\omega_{\underline{d}}(x)\subsetneq
\omega_{\overline{d}}(x)= \omega_{B^*}(x)\subsetneq \omega_f(x) ;$

  \item[Case  (3)    ]  \, \
$\emptyset=\omega_{B_*}(x)\subsetneq  \omega_{\underline{d}}(x)\subsetneq
\omega_{\overline{d}}(x)= \omega_{B^*}(x)= \omega_f(x) ;$

 \item[Case  (3')    ]  \,  \
$\emptyset=\omega_{B_*}(x)\subsetneq  \omega_{\underline{d}}(x)\subsetneq
\omega_{\overline{d}}(x)= \omega_{B^*}(x)\subsetneq \omega_f(x) ;$

  \item[Case  (4)   ]  \,  \   $\emptyset=\omega_{B_*}(x)=  \omega_{\underline{d}}(x) \subsetneq \omega_{\overline{d}}(x)\subsetneq \omega_{B^*}(x)= \omega_f(x);$
  \item[Case  (4')    ]  \,  \   $\emptyset=\omega_{B_*}(x)=  \omega_{\underline{d}}(x) \subsetneq \omega_{\overline{d}}(x)\subsetneq \omega_{B^*}(x)\subsetneq \omega_f(x);$
 \item[Case  (5)    ]  \,  \
$\emptyset=\omega_{B_*}(x)\subsetneq \omega_{\underline{d}}(x) \subsetneq \omega_{\overline{d}}(x)\subsetneq \omega_{B^*}(x)= \omega_f(x);$

 \item[Case  (5')   ]  \, \
$\emptyset=\omega_{B_*}(x)\subsetneq \omega_{\underline{d}}(x) \subsetneq \omega_{\overline{d}}(x)\subsetneq \omega_{B^*}(x)\subsetneq \omega_f(x).$

  \item[Case  (6)    ] \, \ $  \emptyset=\omega_{B_*}(x)\subsetneq\omega_{\underline{d}}(x)= \omega_{\overline{d}}(x)= \omega_{B^*}(x)= \omega_f(x);$
  \item[Case  (6')    ]  \, \ $  \emptyset=\omega_{B_*}(x)\subsetneq\omega_{\underline{d}}(x)= \omega_{\overline{d}}(x)= \omega_{B^*}(x)\subsetneq\omega_f(x).$
  \end{description}
If $\omega_{B_*}(x)\neq \emptyset,$   then the relation  between  $   \omega_{B_*}(x),\,\omega_{\underline{d}}(x),\,\omega_{\overline{d}}(x),\, \omega_{B^*}(x),\,\omega_f(x)$ has sixteen possible cases,  which   are unknown. In this paper, we just consider the systems with $\omega_{B_*}(x)=\emptyset$. A surprising discovery is that strongly distributional chaos does not just occur  in one fixed layer. 
 Here we find there are twelve different layers for which every layer is strongly distributional chaoic, as a refined version of Theorem \ref{thm-NR-new}.

\begin{maintheorem}\label{thm-NR}
Suppose $(X,  f)$ is a mixing expanding map or a transitive Anosov diffeomorphism on a compact manifold. Then  the set   $ \{x\in X\ |\ x\ \text{satisfies Case}\ \mathrm{(i)}\}\cap NR(f)$  is strongly distributional chaoic for any fixed  $
i\in\{1,2,3,4,5,6,1',2',3',4',5',6'\}.$
\end{maintheorem}


\subsection{Combination with Irregular Set and Level Set}
 Let us recall irregular set and level sets which were studied a lot in the dimension theory (or multifractal analysis), for example, see \cite{Barreira-Schmeling2000,BPS,  TDbeta,DOT,Thompson2009,Thompson2008}. Let $\mathcal M(X)$, $\mathcal M_f(X)$, $\mathcal{M}^{e}_{f}(X)$ denote the space of probability measures, $f$-invariant, $f$-ergodic probability measures respectively.

For a continuous function $\varphi$ on $X$, denote the \emph{$\varphi$-irregular set} by
\begin{equation*}
  I_{\varphi}(f) := \left\{x\in X: \lim_{n\to\infty}\frac1n\sum_{i=0}^{n-1}\varphi(f^ix) \,\, \text{ diverges }\right\},
\end{equation*}
Denote:
$$L_\varphi:=\left[\inf_{\mu\in \mathcal M_{f}(X)}\int\varphi d\mu,  \,  \sup_{\mu\in \mathcal M_{f}(X)}\int\varphi d\mu\right],\ Int(L_\varphi):=\left(\inf_{\mu\in \mathcal M_{f}(X)}\int\varphi d\mu,  \,  \sup_{\mu\in \mathcal M_{f}(X)}\int\varphi d\mu\right).  $$

\noindent For any $a\in  L_\varphi,$
denote the level set by
\begin{equation*}
  R_{\varphi}(a) := \left\{x\in X: \lim_{n\to\infty}\frac1n\sum_{i=0}^{n-1}\varphi(f^ix)=a\right\}.
\end{equation*}

The following two theorems are to say that there are no effect on strongly distributional chaos in the twelve layers of Theorem \ref{thm-NR} when they intersect irregular set and level sets. 
\begin{maintheorem}\label{thm-Ir}
Suppose $(X,  f)$ is a mixing expanding map or a transitive Anosov diffeomorphism on a compact manifold. Let $\varphi$ be a continuous function on $X$. If $I_{\varphi}(f)\neq\emptyset$, then  the set $\{x\in X\mid x \text{ satisfies case } (i)\}\cap NR(f)\cap I_{\varphi}(f)$ is strongly distributional chaotic for any $i\in\{1,2,3,4,5,6,1',2',3',4',5',6'\}.$
\end{maintheorem}


\begin{maintheorem}\label{thm-Level}
Suppose $(X,  f)$ is a mixing expanding map or a transitive Anosov diffeomorphism on a compact manifold. Let $\varphi$ be a continuous function on $X$. If $Int(L_{\varphi})\neq\emptyset$, then for any $a\in Int(L_{\varphi})$,   the set $\{x\in X\mid x \text{ satisfies case } (i)\}\cap NR(f)\cap R_\varphi(a)$ is strongly distributional chaotic for any $i\in\{1,2,3,4,5,6,1',2',3',4',5',6'\}.$
\end{maintheorem}

In this paper, we mainly consider continuous maps on a compact metric space. Everything can also be formulated for homeomorphisms on a compact metric space. The dfference is that we must consider expansivity in place of positive expansivity, and shadwoing property of homeomorphisms is defined for bilateral pesudo-orbit rather than unilateral pesudo-orbit.
We leave the details to the readers.

\bigskip

\noindent {\bf Organization of the paper.} Section \ref{Section-Prelim} is a review  of definitions to make precise the statements of the theorems and some basic facts useful for the proof of main results.   In Section \ref{sectionLemma} we give some key technique lemmas 
 for which Lemma \ref{DC1-omega-Lemma} and \ref{strong-elementary-entropy-dense} are the crucial lemmas to prove the main theorems  and
 in Section \ref{sectionEndOfProofs}   we end  the proof of main theorems. In Section \ref{section-application}, we apply the results in the previous sections to more systems, including mixing subshifts of finite type, $\beta$-shifts, homoclinic classes and hyperbolic ergodic measures.

\section{Preliminaries}\label{Section-Prelim}

\subsection{Internally Chain Mixing and Shadowing Property}
Given a dynamical system $(X,f).$
For $l\in\mathbb{N}$, a sequence $\mathfrak{C}^l=\langle x_1,  \cdots,  x_l\rangle$ is called a \emph{chain} with length $l$. Define
$$\delta_{\mathfrak{C}^l}:=\frac{1}{l}\sum_{i=1}^{l}\delta_{x_i}.$$
\noindent For $A\subseteq X$, we say a chain $\mathfrak{C}^l=\langle x_1,  \cdots,  x_l\rangle$ is in $A$ if $\{x_i\}_{i=1}^{l}\subseteq A$.  Furthermore,   if $d(f(x_i),  x_{i+1})<\eps,  1\leq i\leq l-1$,   we call $\mathfrak{C}^l$ an $\eps$-chain with length $l$. We say an $\eps$-chain $\mathfrak{C}^l_{ab}=\langle x_1,  \cdots,  x_l\rangle$ connects $a$ and $b$ if $x_1=a$ and $d(f(x_l),b)<\eps$.  For two $\eps$-chain $\mathfrak{C}^{l_1}_{ab}=\langle x_1,  \cdots,  x_{l_1}\rangle$ and $\mathfrak{C}^{l_2}_{bc}=\langle y_1,  \cdots,  y_{l_2}\rangle$, define
$$\mathfrak{C}^{l_1}_{ab}\mathfrak{C}^{l_2}_{bc}=\langle x_1,  \cdots,  x_{l_1}, y_1,  \cdots,  y_{l_2}\rangle.$$
Obviously, $\mathfrak{C}^{l_1}_{ab}\mathfrak{C}^{l_2}_{bc}$ is $\eps$-chain with length $l_1+l_2$ connecting $a$ and $c$. For $x\in X$, we define
$$orb(x,n):=\langle x, fx, \cdots,  f^{n-1}x\rangle.$$

\subsubsection{Internally Chain Mixing}
\begin{Def}\label{Def-ICM}
Let $A\subseteq X$ be a nonempty closed invariant set. We call $A$ \emph{internally chain transitive} if for any $a,  b\in A$ and any $\eps>0$,   there is an $n\in\mathbb{N}$ and an $\eps$-chain $\mathfrak{C}_{ab}^n$ in $A$ connecting $a$ and $b$. We call $A$ \emph{internally chain mixing} if for any $a,  b\in A$ and any $\eps>0$,   there is an $N\in\mathbb{N}$, such that for any $n\geq N$, there is an $\eps$-chain $\mathfrak{C}_{ab}^n$ in $A$ connecting $a$ and $b$. We denote the collection of internally chain transitive(mixing) sets by $ICT(ICM)$. Obviously $ICM\subseteq ICT$.
\end{Def}

\begin{Lem}\label{omega-subset-ICT}\cite{HSZ-chain-trans}
For any $x\in X$, $\omega_f(x)\in ICT$.
\end{Lem}

\begin{Lem}\label{ICT-in-ICM}
Suppose $\Lambda\in ICT$, $A\in ICM$ and $A\subseteq \Lambda$ . Then $\Lambda\in ICM$.
\end{Lem}

\begin{proof}
Fix a point $x\in A$. Then for any $\eps>0$, there is an $L\in\mathbb{N}$ such that for any $l\geq L$, there is an $\eps$-chain $\mathfrak{C}_{xx}^l$. Note that $\Lambda\in ICT$. Then for any $a,b\in \Lambda$, there are two $\eps$-chain $\mathfrak{C}_{ax}^{l_1}$ and $\mathfrak{C}_{xb}^{l_2}$ for some integers $l_1$ and $l_2.$ Then for any $\hat{l}\geq L+l_1+l_2$, there is an $\eps$-chain $\mathfrak{C}_{ab}^{\hat{l}}=\mathfrak{C}_{ax}^{l_1}\mathfrak{C}_{xx}^l\mathfrak{C}_{xb}^{l_2}$ connecting $a$ and $b$ where $l=\hat{l}-l_1-l_2$.
\end{proof}

\begin{Cor}
Suppose $\Lambda\in ICT$ and $\Lambda$ contains a fixed point. Then $\Lambda\in ICM$.
\end{Cor}

\subsubsection{Shadowing Property and Exponential Shadowing Property}
\begin{Def}
	An infinite sequence $\langle x_n\rangle_{n=1}^{+\infty}$ of points
	in $X$ is a \emph{$\delta$-pseudo-orbit} for a dynamical system $(X,f)$ if $d(x_{n+1},f(x_{n}))<\delta$ for each $n \geq1$. We say that a dynamical system $(X,f)$ has the \emph{shadowing property} if for every $\varepsilon>0$ there is a $\delta >0$ such that any $\delta$-pseudo-orbit $\langle x_n\rangle_{n=1}^{\infty}$ can be $\varepsilon$-shadowed by a point
	$y\in X$, that is $d(f^{n}(y),x_{n})<\varepsilon$ for all $n \geq1$.
\end{Def}

Given $x\in X$ and $i\geq1$, let
\begin{equation*}
	\{x,i\}:=\{f^{j}(x):j=0,1,\dots,i-1\}.
\end{equation*}
For a sequence of points $\langle x_n\rangle_{n=1}^{+\infty}$ in $X$ and a sequence of positive integers $(i_n)_{n=1}^{+\infty}$ we call $\left\{x_{n}, i_{n}\right\}_{n=1}^{+\infty}$ a $\delta$-pseudo-orbit, if
$
d\left(f^{i_{n}}\left(x_{n}\right), x_{n+1}\right)<\delta
$
for all $n\geq1.$ Given $\varepsilon>0$ and $\lambda>0$, we call a point $x\in X$ an \emph{(exponentially) $(\varepsilon,\lambda)$-shadowing} point for a pseudo-orbit $\left\{x_{n}, i_{n}\right\}_{n=1}^{+\infty},$ if
$$
d\left(f^{c_{n}+j}(x), f^{j}\left(x_{n}\right)\right)<\varepsilon \cdot e^{-\min \left\{j, i_n-1-j\right\} \lambda}
$$
for all $0\leq j\leq i_{n}-1$ and $n\geq1$, where $c_{i}$ is defined as
$$
c_{n}=\left\{\begin{array}{ll}
	0, & \text { for } n=1 \\
	\sum_{m=1}^{n-1} i_{m}, & \text { for } n>1.
\end{array}\right.
$$
\begin{Def}
	Let $\lambda>0$. We say that a dynamical system $(X,f)$ has the \emph{exponential shadowing property} with exponent $\lambda$ if for every $\varepsilon>0$ there is a $\delta >0$ such that any $\delta$-pseudo-orbit $\{x_{n},i_{n}\}_{n=1}^{\infty}$ can be $(\varepsilon,\lambda)$-shadowed by a point $x\in X$.
\end{Def}

\begin{Rem}\label{remark-shadowing}
	From the definitions, if a dynamical system $(X,f)$ has exponential shadowing property, then it has shadowing property.
\end{Rem}
\begin{Def}
A sequence $\langle x_i\rangle_{i=1}^{+\infty}$ is called a \emph{limit-pseudo-orbit} if
$$\lim_{i\to\infty} d(f(x_i),  x_{i+1})=0.  $$
Moreover,   $\langle x_i\rangle_{i=1}^{+\infty}$ are \emph{limit-shadowed} by $y\in X$ if
$$\lim_{i\to\infty}d(f^{i-1}y,  x_i)=0.  $$
Then we say $(X,  f)$ has the \emph{limit-shadowing property} if any limit-pseudo-orbit can be limit-shadowed.
\end{Def}
\begin{Def}
For any $\delta>0$,   a sequence $\langle x_i\rangle_{i=1}^{+\infty}$ is called a \textit{$\delta$-limit-pseudo-orbit} if $\langle x_n\rangle_{n=1}^{+\infty}$ is both a $\delta$-pseudo-orbit and a limit-pseudo-orbit.
Furthermore,   $\langle x_n\rangle_{n=1}^{+\infty}$ is \textit{$\eps$-limit-shadowed} by some $y\in X$ if
$\langle x_n\rangle_{n=1}^{+\infty}$ is both $\eps$-shadowed and limit-shadowed by $y$.
Finallly,   we say that $(X,f)$ has the \textit{s-limit-shadowing} property if for any $\eps>0$,   there exists  $\delta>0$ such that any $\delta$-limit-pseudo-orbit can be $\eps$-limit-shadowed.
\end{Def}
\begin{Rem}
	When $f:X\to X$ is a homeomorphism, on a compact metric space, various shadowing properties are defined for $\langle x_n\rangle_{n=-\infty}^{+\infty}$ or $\{x_{n},i_{n}\}_{n=-\infty}^{\infty}.$ For example, see \cite{LSaka2005} for the definitions of shadowing property, limit-shadowing property, see \cite{Tian2015-2} for the  definition of exponential shadowing property.
\end{Rem}

We say that $(X,  f)$ is \emph{positively expansive} if there exists a constant $e>0$ such that for any $x\neq  y\in X$,   $d(f^i(x),  f^i(y))> e$ for some integer $i\geq 0$. We call $e$ the expansive constant.

\begin{Lem}\label{lem-shadowing-s-limit-shadowing} \cite{Sakai}
If $(X,  f)$ is positively expansive and has the shadowing property,
  then $(X,  f)$ has the s-limit shadowing property.
\end{Lem}
\begin{Lem}\label{exponential-s-limit-shadowing}
	If $(X,f)$ is positively expansive and has exponential shadowing property with exponent $\lambda$, then for any $\varepsilon>0$, there exists $\delta>0$ such that for any $\delta$-limit-pseudo-orbit $\left\{x_{n}, i_{n}\right\}_{n=1}^{+\infty}$, there exists $y\in X$ such that $\left\{x_{n}, i_{n}\right\}_{n=1}^{+\infty}$ is both $(\varepsilon,\lambda)$-shadowed and limit-shadowed by $y$.  
\end{Lem}
\begin{proof}
Suppose that the expansive constant is $e$. WLOG, we assume that $0<\varepsilon<\frac{e}{2}$. By Remark \ref{remark-shadowing} and Lemma \ref{lem-shadowing-s-limit-shadowing}, $(X,  f)$ has the s-limit shadowing property, there exists $\delta_1$, such that for any $\delta_1$-limit-pseudo-orbit $\left\{x_{n}, i_{n}\right\}_{n=1}^{+\infty}$, there exists $y_1$ such that $\left\{x_{n}, i_{n}\right\}_{n=1}^{+\infty}$ is both $\varepsilon$-shadowed and limit-shadowed by $y_1$. Since $(X,f)$ has exponential shadowing property with exponent $\lambda$, there exists $\delta_2>0$ such that for any $\delta_2$-limit-pseudo-orbit $\left\{x_{n}, i_{n}\right\}_{n=1}^{+\infty}$, there exists  $y_2\in X$ such that $\left\{x_{n}, i_{n}\right\}_{n=1}^{+\infty}$ is $(\varepsilon,\lambda)$-shadowed by $y_2$. Denote $\delta=\min\{\delta_1,\delta_2\}$, by expansiveness, we finish the proof.
\end{proof}
\subsubsection{Specification Property}

\begin{Def}
Suppose that $f$ is a continuous map on a nondegenerate(i.e,with at least two points) compact metric space ($X$,$d$). We say $X$ has \textit{strong specification property}, if for
any $\varepsilon > 0$, there is a positive integer $K_\varepsilon$ such that for any integer s $\ge$ 2, any set $\{y_1,y_2,\cdots,y_s\}$ of $s$ points of $X$, and any sequence\\
$$0 = a_1 \le b_1 < a_2 \le b_2 < \cdots < a_s \le b_s$$ of 2$s$ integers with $$a_{m+1}-b_m \ge K_\varepsilon$$ for $m = 1,2,\cdots,s-1$, there is a point $x$ in $X$ such that the
following two conditions hold:
\begin{description}
	\item[(a)] $d(f^i(x),f^i(y_m)) < \varepsilon$ for all positive integers m $\le$ s and all integers $i$ with $a_m \le i \le b_m$;
	\item[(b)] $f^n(x) = x$, where $n = b_s + K_\varepsilon$.
\end{description}
If the periodicity condition (b) is omitted, we say that $X$ has \emph{specification property}. Obviously, if $X$ has specification property, $X\in ICM$.
\end{Def}

\begin{Lem}\label{mixing+shadowing=spe}\cite[Proposition 23.20]{Sig}
If $(X,f)$ is positively expansive, mixing and has  the shadowing property, then it has strong specification property.
\end{Lem}
\begin{Lem}\label{lemma-full}\cite[Proposition 21.12]{Sig}
	Suppose that $(X,f)$ is a dynamical system satisfying strong specification property, then the set of invariant measures with full support is residual in $\mathcal{M}_f(X).$
\end{Lem}
\begin{Lem}\label{lemma-ergo}\cite[Proposition 21.9]{Sig}
	Suppose that $(X,f)$ is a dynamical system satisfying strong specification property, then $\mathcal M_f^e(X)$ is residual in $\mathcal{M}_f(X).$
\end{Lem}

\subsection{Anosov diffeomorphisms and expanding maps}

Let $M$ be a compact smooth Riemann manifold without boundary. $f:M\to M$ is a diffeomorphism. A $f$-invariant set $\Lambda\subset M$ is said to be uniformly hyperbolic, if for any $x\in \Lambda$ there is a splitting of the tangent space $T_{x}M=E^{s}(x)\oplus E^{u}(x)$ which is preserved by the differential $Df$ of $f$:
\begin{equation*}
	Df(E^{s}(x))=E^{s}(f(x)),\ Df(E^{u}(x))=E^{u}(f(x)),
\end{equation*}
and there are constants $C>0$ and $0<\lambda<1$ such that for all $n\geq 0$
\begin{equation*}
	|Df^{n}(v)|\leq C\lambda^{n}|v|,\ \forall x\in \Lambda,\ v\in E^{s}(x),
\end{equation*}
\begin{equation*}
	|Df^{-n}(v)|\leq C\lambda^{n}|v|,\ \forall x\in \Lambda,\ v\in E^{u}(x).
\end{equation*}
If $M$ is a uniformly hyperbolic set, then $f$ is called an Anosov diffeomorphism. A hyperbolic set $\Lambda$ is said to be locally maximal for $f$ if there exists a neighborhood $U$ of $\Lambda$ in $M$ such that $\Lambda=\bigcap_{n=-\infty}^{+\infty}f^{n}(U)$.

When $f:X\to X$ is a homeomorphism on a compact metric space,   we say that $(X,  f)$ is \emph{expansive} if there exists a constant $e>0$ such that for any $x,  y\in X$,   $d(f^i(x),  f^i(y))> e$ for some integer $i.$
\begin{Thm}\label{thm-2}
	Every transitive Anosov diffeomorphism on a compact connected manifold or a system restricted on a mixing locally maximal hyperbolic set is expansive, mixing and has the exponential shadowing property.
\end{Thm}
\begin{proof}
	A system restricted on a  locally maximal hyperbolic set has exponential shadowing peoperty by Proposition \cite[Proposition 2.7]{Tian2015-2} and is expansive by \cite[Corollary 6.4.10]{KatHas}.
	 By spectral decomposition, every transitive Anosov diffeomorphism on a compact connected manifold is mixing \cite[Corollary 18.3.5]{KatHas}. Finally every Anosov diffeomorphism is locally maximal, so it has the exponential shadowing property.
\end{proof}

A $C^{1}$ map $f:M\to M$ is said to be expanding if there are constants $C>0$ and $0<\lambda<1$ such that for all $n\geq 0$
\begin{equation*}
	|Df^{n}(v)|\geq C\lambda^{-n}|v|,\ \forall x\in M,\ v\in T_{x}M.
\end{equation*}
\begin{Thm}\label{thm-3}
	Every expanding map on a compact connected manifold is positively expansive and has the exponential shadowing property.
\end{Thm}
\begin{proof}
	Every expanding map  is positively expansive by \cite[Theorem 1.2.1]{AH} and has the exponential shadowing property  by \cite[Propositive 6.1]{HT}. 
\end{proof}

\subsection{Distributional Chaos and Strongly Distributional Chaos}\label{subsection-chaos}
The notion of chaos, as well as Li-Yorke pair and scrambled set, was invented in 1975 by Li and Yorke in the context of continuous transformations of the interval \cite{LY}. Since then various extensions definition of chaos have been developed. One style bases on topological perspective, which specifies how the scrambled set is placed in the space, such as dense chaos \cite{Genric Chaos} and generic chaos \cite{Genric Chaos,S Generic Ghaos}. Another one derives from statistical perspective by adding some statistical restriction to the definition of Li-Yorke pair, which results in distributional chaos.

\begin{Def}
A pair $x,y\in X$ is DC1-scrambled if the following two conditions hold:
$$\forall t>0,\ \limsup_{n\to \infty}\frac{1}{n}|\{i\in [0,n-1]:\ d(f^i(x),f^i(y))<t\}|=1.$$
$$\exists t_0>0,\ \liminf_{n\to \infty}\frac{1}{n}|\{i\in [0,n-1]:\ d(f^i(x),f^i(y))<t_0\}|=0.$$
\end{Def}
In other words, the orbits of $x$ and $y$ are arbitrarily close with upper density one, but for some distance, with lower density zero. A set $S$ is called a DC1-scrambled set if any pair of its distinct points is DC1-scrambled. A map $f$ is called distributional chaotic of type 1(DC1 chaotic for brevity), if there is an uncountable DC1-scrambled set $S\subseteq X$. In this paper, we focus on DC1 chaotic. Readers can refer to \cite{Dwic,SS,SS2} for the definition of DC2 and DC3 if necessary.

Now we recall from \cite{HT} a kind of chaos, strongly distributional chaos, which is stronger than usual distributional chaos and Li-Yorke chaos.
For any positive integer $n$, points $x,y \in M$ and $t \in \mathbb{R}$ let
\begin{equation*}
	\Phi _{xy}^{(n)}(t,f)=\frac{1}{n}|\{0\leq i \leq n-1:d(f^{i}x,f^{i}y)<t\}|,
\end{equation*}
where $|A|$ denotes the cardinality of the set $A$. Let us denote by $\Phi _{xy}$ the following function:
\begin{equation*}
	\Phi _{xy}(t,f)=\liminf_{n \to \infty}\Phi _{xy}^{(n)}(t,f).
\end{equation*}
Define $\mathcal{A}=\{\alpha(\cdot):\alpha\ \text{is a nondecreasing map form}\  \mathbb{N}\ \text{to}\ [0,+\infty),\ \lim\limits_{n\to \infty}\alpha(n)=+\infty\ \text{and}\  \lim\limits_{n\to\infty}\frac{\alpha(n)}{n}=0\}$.
For any positive integer $n$, points $x,y \in M$, $t \in \mathbb{R}$ and $\alpha\in\mathcal{A},$ let
\begin{equation*}
	\Phi _{xy}^{(n)}(t,f,\alpha)=\frac{1}{n}|\{1\leq i \leq n:\sum_{j=0}^{i-1}d(f^{j}x,f^{j}y)<\alpha(i)t\}|.
\end{equation*}
Let us denote by $\Phi _{xy}^{*}(t,f,\alpha)$ the following functions:
\begin{equation*}
	\Phi _{xy}^{*}(t,f,\alpha)=\limsup_{n \to \infty}\Phi _{xy}^{(n)}(t,f,\alpha).
\end{equation*}
\begin{Def}
	A pair $x,y\in X$ is $\alpha$-DC1-scrambled if the following two conditions hold:
	\begin{equation*}
		\Phi _{xy}(t_{0},f)=0\ \mathrm{for}\ \mathrm{some}\ t_{0}>0\ \mathrm{and}
	\end{equation*}
	\begin{equation*}
		\Phi _{xy}^{*}(t,f,\alpha)=1\ \mathrm{for}\ \mathrm{all}\ t>0.
	\end{equation*}
	A set $S$ is called a $\alpha$-DC1-scrambled set if any pair of distinct points in $S$ is $\alpha$-DC1-scrambled.
	A subset $Y\subset M$ is said to be strongly distributional chaotic if it has an uncountable $\alpha$-DC1-scrambled set for any $\alpha\in\mathcal{A}$. 
\end{Def}

\begin{Rem}
	Strongly distributional chaos  is stronger than  distributional chaos of type 1 (see \cite[Proposition 2.5]{HT}).
\end{Rem}

\subsection{Basic facts for statistical $\omega$-sets}
For any $\mu\in \mathcal M_f(X)$, we denote $S_\mu=supp(\mu)=\{x\in X|\ \mu(U)>0\ $for any neighborhood $U$ of $x\}$ the support of $\mu.$ Given $x\in X$, denote $V_f(x)\subseteq \mathcal M_f(X)$ the set of all accumulation points of the empirical measures
$$
\mathcal{E}_n (x):=\frac1{n}\sum_{i=0}^{n-1}\delta_{f^i (x)},
$$
where $\delta_x$ is the Dirac measure concentrate on $x$. As is known,   $V_f(x)$ is  a non-empty compact connected subset of $\mathcal M_f(X)$ \cite{Sig}. {For any two positive integers $a_k<b_k$,   denote
$[a_k,  b_k]=\{a_k,  a_k+1,  \cdots,  b_k\}$ and $[a_k,  b_k)=[a_k,  b_k-1],  (a_k,  b_k)=[a_k+1,  b_k-1],  (a_k,  b_k]=[a_k+1,  b_k]$.
A point $x$ is called {\it quasi-generic} for some measure $\mu,  $ if there is a sequence of positive integer intervals $I_k=[a_k,  b_k)$ with $b_k-a_k\to\infty$ such that
$$\lim_{k\rightarrow\infty}\frac{1}{b_k-a_k}\sum_{j=a_k}^{b_k-1}\delta_{f^j(x)}=\mu$$
in weak$^*$ topology.
}
Let $V_f^*(x)=\{\mu\in M(f,  X): \,  x \text{ is quasi-generic for } \mu\}$.
 This concept is from \cite{Fur} and from there it is known $V_f^*(x)$ is always nonempty,
 compact and connected.   Note that $V_f(x)\subseteq V_f^*(x).  $

\begin{Prop}\cite[Theorem 1.4]{DT}\label{density-measure-prop}
Suppose $(X,  f)$ is a topological dynamical system. 
\begin{description}
  \item[(1)] For any $x\in X,   $ $\omega_{\underline{d}}(x)= \bigcap_{\mu\in V_f(x)} S_\mu$.
  \item[(2)] For any $x\in X,   $  $\omega_{\overline{d}}(x)=\overline{\bigcup_{\mu\in V_f(x)} S_\mu}\neq \emptyset$.
  \item[(3)] For any $x\in X,   $  $\omega_{B_*}(x)= \bigcap_{\mu\in V^*_f(x)} S_\mu =  {\bigcap_{\mu\in \mathcal M_f(\omega_f(x))}S_{\mu}}= {\bigcap_{\mu\in \mathcal M_f^e(\omega_f(x))}S_{\mu}}.$ If
   $\omega_{B_*}(x)\neq \emptyset$,   then $ \omega_{B_*}(x)$ is minimal.
  \item[(4)] For any $x\in X,   $ $\omega_{B^*}(x)= \overline{\bigcup_{\mu\in V^*_f(x)} S_\mu}=\overline{\bigcup_{\mu\in \mathcal M_f(  \omega_f(x))}S_{\mu}}=\overline{\bigcup_{\mu\in \mathcal M_f^e(\omega_f(x))}S_{\mu}}\neq \emptyset;$

\item[(5)] For any invariant measure $\mu$ and $\mu$ a.  e.   $x\in X,   $
  $$ \omega_{\underline{d}}(x)= \omega_{\overline{d}}(x)= \omega_{B^*}(x)= \omega_f(x)=S_\mu.  $$

\end{description}

 \end{Prop}

\section{Technique Lemmas}\label{sectionLemma}
\subsection{Distributional Chaos in Saturated Sets}
Consider a dynamical system $(X,f).$ Given $x\in X$, it is known that  $V_f(x)$ is  a non-empty compact connected subset of $\mathcal M_f(X)$ \cite{Sig}. So for any non-empty compact connected subset $K$ of $\mathcal M_f(X)$,   it is logical to define the following set
$$G_K:=\{x\in X~|~V_f(X)=K\}.  $$
$G_K$ is known as the saturated set of $K$. Particularly,   if $K=\{\mu\}$ for some ergodic measure $\mu$,   then $G_{\mu}$ is just the generic points of $\mu$.   The existence of saturated sets are studied by Sigmund in \cite{SigSpe}. The Bowen entropy of saturated sets are studied by Pfister and Sullivan in \cite{PS2}. 
Here we consider distributional chaos in saturated sets. Let $\Lambda\subseteq X$ be a closed invariant subset and $K$ is a non-empty compact connected subset of $\mathcal M_f(\Lambda)$. Define
$$G_K^{\Lambda}:=G_K\cap\{x\in X~|~\omega_f(x)=\Lambda\}.  $$
We say a pair $p,q\in X$ is {\it distal} if $$\liminf_{i\to\infty}d(f^i(p),f^i(q))>0.$$ Obviously, $\inf\{d(f^i(p),f^i(q))\mid i\in\mathbb{N}\}>0$ if $p,q$ is distal. We say a subset $M\subseteq X$ has distal pair if there are distinct $p,q\in M$ such that $p,q$ is distal.
\begin{mainlemma}\label{DC1-omega-Lemma}
	Suppose that $(X,f)$ is positively expansive and  transitive, $\Lambda\in ICM$. Let $K\subseteq\mathcal M_f(\Lambda)$ be  non-empty compact connected. If there is a $\mu\in K$ such that $\mu=\theta\mu_1+(1-\theta)\mu_2\ (\mu_1=\mu_2\mathrm{\ are\ legal})$, where $\theta\in[0,1]$, and $G_{\mu_1}$, $G_{\mu_2}$ both have distal pair. Then
	\begin{description}
		\item[(a)] if $(X,f)$ has the shadowing property, then for any non-empty open set $U\subseteq X$, there is an uncountable DC1-scrambled set $S\subseteq G_K^{\Lambda}\cap U$. Particularly, $S\subseteq G_K^{\Lambda}\cap NR$ when $U\subseteq X\setminus \Lambda$;
		\item[(b)] if further $(X,f)$ has exponential shadowing property, then for any $\alpha\in\mathcal{A}$ and for any non-empty open set $U\subseteq X$ there is an uncountable $\alpha$-DC1-scrambled set $S^{\alpha}\subseteq G_K^{\Lambda}\cap U.$ 
	\end{description}
\end{mainlemma}

\subsubsection{Some Lemmas}
We write $\mathbb{N}=\{0,1,2,\cdots\}$ and $\mathbb{N}^+=\{1,2,\cdots\}$.  The cardinality of a finite set $\Lambda$ is denoted by $|\Lambda|$. We set
$$
\langle f,\mu \rangle\ :=\ \int_Xfd\mu.
$$
There exists a countable and separating set of continuous function $\{f_1,f_2,\cdots\}$ with $0\leq f_k(x)\leq 1$, and such that
$$
d(\mu,\nu)\ :=\ \parallel\mu-\nu\parallel\ :=\ \sum_{k\geq 1}2^{-k}\mid\langle f_k,\mu-\nu \rangle\mid
$$
defines a metric for the weak*-topology on $\mathcal{M}_f(x)$. We refer to \cite{PS2} and use the metric on $X$ as following defined by Pfister and Sullivan.
$$
d(x,y) := d(\delta_x,\delta_y),
$$
which is equivalent to the original metric on $X$. Readers will find the benefits of using this metric in our proof later.
\begin{Lem}\label{measure distance}
For any $\varepsilon > 0,\delta >0$ and two sequences $\{x_i\}_{i=0}^{n-1},\{y_i\}_{i=0}^{n-1}$ of $X$ such that $d(x_i,y_i)<\varepsilon$ holds for any $i\in [0,n-1]$, then for any $J\subseteq \{0,1,\cdots,n-1\}$, $\frac{n-|J|}{n}<\delta$, one has:
\begin{description}
	\item[(a)] $d(\frac{1}{n}\sum_{i=0}^{n-1}\delta_{x_i},\frac{1}{n}\sum_{i=0}^{n-1}\delta_{y_i})<\varepsilon.$
	\item[(b)] $d(\frac{1}{n}\sum_{i=0}^{n-1}\delta_{x_i},\frac{1}{|J|}\sum_{i\in J}\delta_{y_i})<\varepsilon+2\delta.$
\end{description}
\end{Lem}

Lemma \ref{measure distance} is easy to be verified and shows us that if any two orbit of $x$ and $y$ in finite steps are close in the most of time, then the two empirical measures induced by $x,y$ are also close.

Given a metric space $(X,d).$  We denote the Hausdorff distance between two nonempty subsets of $X,$ $A$ and $B,$  by $$d_H(A,  B):=\max\set{\sup_{x\in A}\inf_{y\in B}d(x,  y),\sup_{y\in B}\inf_{x\in A}d(y, x)  }.$$ A point $x \in X$ is $periodic$ if there is a $n\in \mathbb{N}$ such that $f^n(x)=x$. We denote $Per$ the set of periodic points. We denote $Per_n:=\{x\in X\mid f^nx=x\}$.
\begin{Lem}\label{lemma-single}
Suppose that $(X,f)$ is positively expansive and has the shadowing property, $\Lambda\in ICM$. Then for any $\eps>0$, any $x\in\Lambda$, any $\mu\in\mathcal M_f(\Lambda)$ and its neighborhood $F_{\mu}$, there exists an $M\in \mathbb{N}$ such that for any $n>M$, there exists $p\in B(x,\eps)$ satisfying
\begin{description}
\item[(a)] $p\in Per_n,\ \ \mathcal E_n(p)\in F_\mu$;

\item[(b)] $d_H(orb(p), \Lambda)<\eps$.
\end{description}

\end{Lem}

\begin{proof}
Since $F_\mu$ is a neighborhood of $\mu$, there is an $a>0$ such that $B(\mu,a)\subseteq F_\mu$. By the ergodic decomposition Theorem, there exists a finite convex combination of ergodic measure $\sum_{i=1}^mc_i\nu_i\in B(\mu, a/5)$. Moreover, by the denseness of rational numbers, we can choose each $c_i=\frac{b_i}{b}$ with $b_i\in\mathbb{N}$ and $\sum_{i=1}^mb_i=b.$ It is known that $\nu_i(G_{\nu_i})=1$. So we can choose $\{p_i\}_{i=1}^{m}\subseteq \Lambda$ with $p_i\in G_{\nu_i},\ i=1,2,\cdots,m.$  Then there is a common $N_c\in\mathbb{N}$ such that for any $n>N_c$ and any $i\in\{1,2,\cdots,m\}$, $\mathcal E_n(p_i)\in B(\nu_i,a/5)$.

Let $\eps^*=\mathrm{min}\{\frac{\eps}{2},\frac{a}{5}, \frac{e}{2}\},$ where $e$ is the expansive constant. By shadowing property, there is a $\delta\in(0,\eps^*)$ such that any {$\delta$-pseudo-orbit} can be {$\eps^*$-shadowed} by some point in $X$. Since $\Lambda$ is compact we can choose finite ball open balls $\{B(x_i,\delta)\}_{i=1}^s$ which covers $\Lambda$ with $\{x_i\}_{i=1}^s\subseteq \Lambda$. Note that $\Lambda\in ICM$, then there exists an $L\in\mathbb{N}$ such that for any $l\geq L$, there is a $\delta$-chain $\mathfrak{C}_{yz}^l$ connecting $y$ and $z$ for any $y,z\in\{x_i\}_{i=1}^s\cup\{x\}\cup\{p_i\}_{i=1}^m.$ For any $p\in\Lambda$, we define
$$x_{(p,n)}:=x_{\mathrm{min}\{i\in[1,s]|f^n(p)\in B(x_i,\delta)\}}.$$
\noindent (Taking the minimum index here has no explicit meanings. We just want to fix a point in $\{x_i\}_{i=1}^s$.)

\noindent Now, let $k\in\mathbb{N}$ large enough such that
\begin{equation}\label{single-condition-1}
kb_i\geq N_c,\ \ 1\leq i \leq m; \ \ \ \text{and}
\end{equation}
\begin{equation}\label{single-condition-2}
\frac{(s+m+1)L+b}{(s+m+1)L+kb+b}<\frac{a}{5}.
\end{equation}

\noindent Let $M=(s+m+1)L+bk$, then for any $n>M$, $n=M+tb+c,$ where $c<b,t\in\mathbb{N}.$ Here we construct $\mathfrak{C}=\mathfrak{C}_{xx}^{n}\mathfrak{C}_{xx}^{n}\cdots$, where
$$\mathfrak{C}_{xx}^{n}=\mathfrak{C}_{xx_1}^{L+c}\mathfrak{C}_{x_1x_2}^{L}\cdots\mathfrak{C}_{x_{s-1}x_s}^{L}\mathfrak{C}_{x_sp_1}^{L}orb(p_1,(k+t)b_1)\mathfrak{C}_{x_{(p_1,(k+t)b_1)}p_2}^{L}orb(p_2,(k+t)b_2)$$
$\ \ \ \ \  \ \ \ \ \ \ \ \ \ \ \ \ \mathfrak{C}_{x_{(p_2,(k+t)b_2)}p_3}^{L}\cdots orb(p_m,(k+t)b_m)\mathfrak{C}_{x_{(p_m,(k+t)b_m)}x}^{L}.$

\noindent It is easy to check $\mathfrak{C}$ is a {$\delta$-pseudo-orbit}. By shadowing property, $\mathfrak{C}$ can be {$\eps^*$-shadowed} by a point $p\in X$. Note that $\mathfrak{C}$ is $n$ periodic. So $d(f^i(p),f^{i+n}(p))<\eps^*+\eps^*=2\eps^*<e$ for any $i\in\mathbb{N}$, which implies $f^n(p)=p$ by expansiveness. By (\ref{single-condition-1}), (\ref{single-condition-2}) and Lemma \ref{measure distance}, we have
\begin{align*}
d(\mathcal E_n(p),\sum_{i=1}^mc_i\nu_i) &\leq d(\mathcal E_n(p),\delta_{\mathfrak{C}_{xx}^n})+d(\delta_{\mathfrak{C}_{xx}^n},\sum_{i=1}^mc_i\nu_i)\\
&\leq \eps^*+d(\delta_{\mathfrak{C}_{xx}^n},\sum_{i=1}^mc_i\nu_i)\\
&\leq \eps^*+d(\delta_{\mathfrak{C}_{xx}^n},\sum_{i=1}^{m}\frac{b_i}{b}\delta_{orb(p_i,(k+t)b_i)})+d(\sum_{i=1}^{m}\frac{b_i}{b}\delta_{orb(p_i,(k+t)b_i)},\sum_{i=1}^mc_i\nu_i)\\
&< \eps^*+2*\frac{a}{5}+\frac{a}{5}\\
&< \frac{4a}{5}.
\end{align*}

\noindent Thus $\mathcal E_n(p)\subseteq B(\mu,a)\subseteq F_\mu$.

Note that $\mathfrak{C}_{xx}^n$ contains $\{x_i\}_{i=1}^{s}$ and $\{B(x_i,\delta)\}_{i=1}^{s}$ covers $\Lambda$. So $\{B(f^i(p),\delta+\eps^*)\}_{i=0}^{n-1}$ covers $\Lambda$. On the other side, $\{f^i(p)\}_{i=0}^{n-1}\subseteq B(\Lambda,\eps^*)$ and $p\in Per_n$. So
$d_H(orb(p),\Lambda)<\eps^*+\delta<2\eps^*\leq \eps.$
\end{proof}

\begin{Lem}\label{lemma-double}
Suppose that $(X,f)$ is positively expansive and has the shadowing property, $\Lambda\in ICM$.  Suppose there are $\mu_1,\mu_2\in\mathcal M_f(\Lambda)$ such that $G_{\mu_1},G_{\mu_2}$ has distal pair $(g_1,e_1),(g_2,e_2)$ respectively. Let $$\zeta=\mathrm{min}\{\inf\{d(f^i(g_1),f^i(e_1))|\ i\in\mathbb{N}\},\inf\{d(f^i(g_2),f^i(e_2))|\ i\in\mathbb{N}\}\}.$$ Then for any $\tau>0$, any $\eps\in(0,\zeta/2)$, any $\theta\in[0,1]$ and any $x\in\Lambda$, there exists an $M\in \mathbb{N}$ such that for any $n>M$, there exists $p^1,p^2\in B(x,\eps)$ satisfying
\begin{description}
\item[(a)] $p^i\in Per_n,\ \ \mathcal E_n(p^i)\in B(\theta\mu_1+(1-\theta)\mu_2,\eps),\ i=1,2$;

\item[(b)] $d_H(orb(p^i), \Lambda)<\eps,\ i=1,2$;

\item[(c)] $\frac{|\ \{i\in[0,n-1]: d(f^i(p^1),f^i(p^2))<\zeta-\eps\}\ |}{n}<\tau$.
\end{description}

\end{Lem}

\begin{proof}
By the denseness of rational numbers, we can assume $r_1,r_2\in\mathbb{N}, r=r_1+r_2$ and $\theta=\frac{r_1}{r}$. Since $g_i,e_i\in G_{\mu_i},i=1,2$, there is a common $N_c\in\mathbb{N}$ such that for any $n\geq N_c$, $\mathcal E_n(g_i)\in B(\mu_i,\eps/4)$ and $\mathcal E_n(e_i)\in B(\mu_i,\eps/4),i=1,2$. Let $\eps^*=\mathrm{min}\{\frac{\eps}{2},\frac{e}{2}\}$. By shadowing property, there is a $\delta\in(0,\eps^*)$ such that any {$\delta$-pseudo-orbit} can be {$\eps^*$-shadowed} by some point in $X$. Since $\Lambda$ is compact we can choose finite ball open balls $\{B(x_i,\delta)\}_{i=1}^s$ which covers $\Lambda$ with $\{x_i\}_{i=1}^s\subseteq \Lambda$. Note that $\Lambda\in ICM$, then there exists an $L\in\mathbb{N}$ such that for any $l\geq L$, there is a $\delta$-chain $\mathfrak{C}_{yz}^l$ connecting $y$ and $z$ for any $y,z\in\{x_i\}_{i=1}^s\cup\{x\}\cup\{g_i\}_{i=1}^2\cup\{e_i\}_{i=1}^2.$

Now, let $k\in\mathbb{N}$ large enough such that
\begin{equation}\label{double-condition-1}
kr_i\geq N_c,\ i=1,2; \ \ \ \ \ \text{and}
\end{equation}
\begin{equation}\label{double-condition-2}
\frac{(s+3)L+r}{(s+3)L+kr+r}<\mathrm{min}\{\tau,\frac{\eps}{4}\}.
\end{equation}

\noindent Let $M=(s+3)L+kr$, then for any $n>M$, $n=M+tr+c,$ where $c<r,t\in\mathbb{N}.$ Here we construct $\mathfrak{C}_1=\mathfrak{C}_{xx,1}^{n}\mathfrak{C}_{xx,1}^{n}\cdots$ and $\mathfrak{C}_2=\mathfrak{C}_{xx,2}^{n}\mathfrak{C}_{xx,2}^{n}\cdots$, where
$$\mathfrak{C}_{xx,1}^{n}=\mathfrak{C}_{xx_1}^{L+c}\mathfrak{C}_{x_1x_2}^{L}\cdots\mathfrak{C}_{x_{s-1}x_s}^{L}\mathfrak{C}_{x_sg_1}^{L}orb(g_1,(k+t)r_1)\mathfrak{C}_{x_{(g_1,(k+t)r_1)}g_2}^{L}orb(g_2,(k+t)r_2)\mathfrak{C}_{x_{(g_2,(k+t)r_2)}x}^{L};$$
$$\mathfrak{C}_{xx,2}^{n}=\mathfrak{C}_{xx_1}^{L+c}\mathfrak{C}_{x_1x_2}^{L}\cdots\mathfrak{C}_{x_{s-1}x_s}^{L}\mathfrak{C}_{x_se_1}^{L}orb(e_1,(k+t)r_1)\mathfrak{C}_{x_{(e_1,(k+t)r_1)}e_2}^{L}orb(e_2,(k+t)r_2)\mathfrak{C}_{x_{(e_2,(k+t)r_2)}x}^{L}.$$

\noindent It is easy to check $\mathfrak{C}_1,\mathfrak{C}_2$ are both {$\delta$-pseudo-orbit}. By shadowing property, $\mathfrak{C}_i$ can be {$\eps^*$-shadowed} by a point $p^i\in X,i=1,2$. With the similar analysis in the proof of Lemma \ref{lemma-single}, item $\mathbf{(a)}$ and $\mathbf{(b)}$ are satisfied. Observing $\mathfrak{C}_{xx,1}^{n}$ and $\mathfrak{C}_{xx,2}^{n}$, one can find that
\begin{equation}\label{distance-1}
d(f^{j+(s+1)L+c}(p^1),f^j(g_1))<\eps^*,\ j\in[0,(k+t)r_1-1],
\end{equation}
\begin{equation}\label{distance-2}
d(f^{j+(s+2)L+c+(k+t)r_1}(p^1),f^j(g_2))<\eps^*,\ j\in[0,(k+t)r_2-1];
\end{equation}
\noindent and
\begin{equation}\label{distance-3}
d(f^{j+(s+1)L+c}(p^2),f^j(e_1))<\eps^*,\ j\in[0,(k+t)r_1],
\end{equation}
\begin{equation}\label{distance-4}
d(f^{j+(s+2)L+c+(k+t)r_1}(p^2),f^j(e_2))<\eps^*,\ j\in[0,(k+t)r_2-1].
\end{equation}

\noindent Note that $\zeta=\mathrm{min}\{\inf\{d(f^i(g_1),f^i(e_1))|\ i\in\mathbb{N}\},\inf\{d(f^i(g_2),f^i(e_2))|\ i\in\mathbb{N}\}\}$. So, with the combination of (\ref{distance-1}), (\ref{distance-2}), (\ref{distance-3}) and (\ref{distance-4}), one has
\begin{equation}\label{distance-result-1}
d(f^{j+(s+1)L+c}(p^1),f^{j+(s+1)L+c}(p^2))>\zeta-2\eps^*\geq\zeta-\eps,\ j\in[0,(k+t)r_1-1],
\end{equation}
\noindent and
\begin{equation}\label{distance-result-2}
d(f^{j+(s+2)L+c+(k+t)r_1}(p^1),f^{j+(s+2)L+c+(k+t)r_1}(p^2))>\zeta-2\eps^*\geq\zeta-\eps,\ j\in[0,(k+t)r_2-1].
\end{equation}
\noindent By (\ref{double-condition-2}), we have
\begin{equation}\label{double-guji-result}
\frac{(k+t)r}{n}>1-\tau.
\end{equation}
Combining (\ref{distance-result-1}), (\ref{distance-result-2}) and (\ref{double-guji-result}), one has $\frac{|\ \{i\in[0,n-1]:d(f^i(p^1),f^i(p^2))<\zeta-\eps\}\ |}{n}<\tau$.
\end{proof}

\subsubsection{Proof of Lemma \ref{DC1-omega-Lemma}}
\textbf{(a):} By the proof of \cite[Theorem 5.1]{PS2}, there is a sequence $\{\alpha_i\}_{i=1}^{\infty}$ in $K$ such that for any $n\in\mathbb{N}$, $\overline{\{\alpha_i:i\in\mathbb{N},i>n\}}=K$ and
\begin{equation}\label{equation-A}
	\lim_{i\to\infty}d(\alpha_i,\alpha_{i+1})=0.
\end{equation}
\noindent For any $\eps>0$ and any $t\in\mathbb{N}^+$, there exists a sequence $\{\beta_i\}_{i=1}^{l}\subseteq K$ such that $\beta_1=\mu,\beta_l=\alpha_t$ and $d(\beta_i,\beta_{i+1})<\eps,i=1,2,\cdots,l-1$ since $K$ is connected. So we can assume that such sequence $\{\alpha_i\}_{i=1}^{\infty}$ has a subsequence $\{\alpha_{i_k}\}_{k=1}^{\infty}$ satisfying
$$i_{k+1}-i_k\geq 2,\ \text{and}\ \alpha_{i_k}=\alpha_{i_k+1}=\mu\ \text{for\ any}\ k\in\mathbb{N}^+.$$
\noindent(If not, add the sequence $\{\beta_i\}_{i=1}^{l}$ to the original sequence.) Define the index set
$$S_1:=\{i_k\ |\ k\in\N\ \},\ \ \ S_2:=\mathbb{N}^+\setminus S_1.$$
Let $(g_1,e_1)$, $(g_2,e_2)$ be the distal pair of $G_{\mu_1}$, $G_{\mu_2}$ respectively and $$\zeta=\mathrm{min}\{\inf\{d(f^i(g_1),\\
f^i(e_1))|\ i\in\mathbb{N}\},\inf\{d(f^i(g_2),f^i(e_2))|\ i\in\mathbb{N}\}\}.$$ Fix a point $z\in U$, there is a $\rho>0$ such that $B(z,\rho)\subseteq U$. Let $\eps_1=\mathrm\{\rho/2,\zeta\}$ and $\eps_{i+1}=\eps_i/2$ for $i\geq 1$. 
By Lemma \ref{lem-shadowing-s-limit-shadowing}, $(X,  f)$ has the s-limit shadowing property.
Then there is a $\delta_i\in(0,\eps_i)$ such that 
any $\delta_i$-limit-pseudo-orbit can be $\eps_i$-limit-shadowed by some point in $X,$ and $\lim_{i\to\infty}\delta_i=0.$
Now, fix a point $z_{\Lambda}\in\Lambda$. There is an $m\in\mathbb{N}$ and a $p_0\in B(z,\eps_1)$ such that $f^m(p_0)\in B(z_{\Lambda},\delta_1/2)$ since $(X,f)$ is transitive. Let $\{\tau_i\}_{i=1}^{\infty}$ be a strictly decreasing sequence with $\lim_{i\to\infty}\tau_i=0$.

For any $i\in S_2$, using Lemma \ref{lemma-single} on $\delta_i/2,z_{\Lambda},\alpha_i,B(\alpha_i,\delta_i/2)$,  we get a positive integer sequence $\{M_i\}_{i\in S_2}$. For any $i\in S_1$, using Lemma \ref{lemma-double} on $\tau_i,\delta_i/2,\theta,z_{\Lambda}$, we get a positive integer sequence $\{M_i\}_{i\in S_1}$.
Let $\{n_i\}_{i=1}^{\infty}$ be an integer sequence with
$$n_i>M_i,\ i\geq 1.$$
Then for $i\in S_2$, we get $p_i\in Per_{n_i}$ by Lemma \ref{lemma-single}. For $i\in S_1$, we get $p_i^1\in Per_{n_i}$ and $p_i^2\in Per_{n_i}$ by Lemma \ref{lemma-double}. Choosing a strictly increasing integer sequence $\{N_i\}_{i=1}^{\infty}$ with
\begin{equation}\label{PS-N-1}
n_{i+1}\leq \tau_i\sum_{j=1}^{i}n_jN_j,\ \ \ \ \text{and}
\end{equation}
\begin{equation}\label{PS-N-2}
\sum_{j=1}^{i-1}n_jN_j\leq\tau_i\sum_{j=1}^{i}n_jN_j.
\end{equation}

\noindent For any $\xi=\{\xi_1,\xi_2,\cdots\}\in\{1,2\}^{\infty}$, we denote
$$\mathfrak{C}(\xi)=orb(p_0,m)\underbrace{\mathfrak{C}_1\cdots\mathfrak{C}_1}_{N_1}\underbrace{\mathfrak{C}_2\cdots\mathfrak{C}_2}_{N_2}\underbrace{\mathfrak{C}_3\cdots\mathfrak{C}_3}_{N_3}\cdots,$$
\noindent where
\begin{equation}\label{pesudo-construction}
\mathfrak{C}_i=
\left\{
             \begin{array}{lr}
             orb(p_i,n_i)\ \text{if}\ i\in S_2;\\
             orb(p_i^{\xi_{[k]}},n_i)\ \text{if}\ i=i_k\ \text{for some}\ k.
             \end{array}
\right.
\end{equation}
\noindent where $[k]:=\text{the\ minimum\ positive\ integer\ of}\ \{k-\sum_{i=0}^{t}i\}_{\{t\in\mathbb{N}\}}$. One can check that $\mathfrak{C}(\xi)$ is a {$\delta_1$-limit-pseudo-orbit}. So there is a point $S_{\mathfrak{C}(\xi)}\in X$ which $\eps_1$-limit-shadows $\mathfrak{C}(\xi)$. Denote $$S:=\bigcup_{\xi\in\{1,2\}^{\infty}}S_{\mathfrak{C}(\xi)}.$$
We complete this proof by proving the following four facts:
\begin{description}
	\item[(1)] $S\subseteq G_K$;
	\item[(2)] For any $y\in S$, $\omega_f(y)=\Lambda$;
	\item[(3)] For any distinct $x,y\in S$, $x,y$ is a DC1-scrambled pair;
	\item[(4)] $S_{\mathfrak{C}(\xi)}\neq S_{\mathfrak{C}(\eta)}$ if $\xi\neq \eta$, which implies $S$ is uncountable.
\end{description}

(1): The method in the proof of item (1) is mainly referring to \cite{PS2}. Firstly, we define two stretched sequences $\{n_l^{\prime}\}_{l=1}^{\infty}$ by
$$n_l^{\prime}=n_k\ \ \text{if}\ \sum_{j=1}^{k-1}N_j+1\leq l \leq\sum_{j=1}^{k}N_j,$$
and $\{\alpha_l^{\prime}\}_{l=1}^{\infty}$ by
$$\alpha_l^{\prime}:=\alpha_1\ \ \text{if}\ l\leq m;$$
$$\alpha_l^{\prime}:=\alpha_k\ \ \text{if}\ \sum_{j=1}^{k-1}n_jN_j+1\leq l-m \leq\sum_{j=1}^{k}n_jN_j.$$
\noindent The sequence $\{\alpha_l^{\prime}\}$ has the same limit point set as the sequence $\{\alpha_k\}$. If for any $y\in S$,
$$\lim_{n\to\infty}d(\mathcal E_n(y),\alpha_n^{\prime})=0,$$
then the two sequence $\{\mathcal E_n(y)\}$ and $\{\alpha_n^{\prime}\}$ have the same limit point set and thus $S\subseteq G_K$. Let $M_k:=m+\sum_{i=1}^{k}n_i^{\prime}$. Because of (\ref{PS-N-1}) and the definition of $\{\alpha_m^{\prime}\}$, it is sufficient to show that
$$\lim_{k\to\infty}d(\mathcal E_{M_k}(y),\alpha_{M_k}^{\prime})=0.$$
Suppose that $M_k=m+\sum_{l=1}^{j}n_lN_l+n_{j+1}\mathfrak{N}$ with $i\in\N$ and $1\leq \mathfrak{N}\leq N_{j+1}$, hence $\alpha_{M_k}^{\prime}=\alpha_{j+1}.$ By  (\ref{PS-N-2}), we have
\begin{equation}\label{G_K-equation-1}
	\begin{split}
		d(\mathcal E_{M_k}(y),\alpha_{M_k}^{\prime}) \leq &\frac{m+\sum_{l=1}^{j-1}n_lN_l}{M_k}d(\mathcal E_{m+\sum_{l=1}^{j-1}n_lN_l}(y),\alpha_{M_k}^{\prime})\\
		&+\frac{n_jN_j}{M_k}d(\mathcal E_{n_jN_j}(f^{m+\sum_{l=1}^{j-1}n_lN_l}(y)),\alpha_{M_k}^{\prime})\\
		&+\frac{n_{j+1}\mathfrak{N}}{M_k}d(\mathcal E_{n_{j+1}\mathfrak{N}}(f^{m+\sum_{l=1}^{j}n_lN_l}(y)),\alpha_{M_k}^{\prime})\\
		\leq &\tau_j+d(\mathcal E_{n_jN_j}(f^{m+\sum_{l=1}^{j-1}n_lN_l}(y)),\alpha_{j+1})+d(\mathcal E_{n_{j+1}\mathfrak{N}}(f^{m+\sum_{l=1}^{j}n_lN_l}(y)),\alpha_{j+1}).
	\end{split}
\end{equation}
By Lemma \ref{lemma-single} and Lemma \ref{lemma-double}, we have 
\begin{equation}\label{G_K-equation-2}
	\begin{split}
		d(\mathcal E_{n_jN_j}(f^{m+\sum_{l=1}^{j-1}n_lN_l}y),\alpha_{j+1})&\leq d(\mathcal E_{n_jN_j}(f^{m+\sum_{l=1}^{j-1}n_lN_l}y),\sum_{i=1}^{N_j}\frac{1}{N_j}\mathcal E_{n_j}(p_j))\\
		&\ \ \ \ +d(\sum_{i=1}^{N_j}\frac{1}{N_j}\mathcal E_{n_j}(p_j),\alpha_j)+d(\alpha_j,\alpha_{j+1})\\
		&\leq d(\mathcal E_{n_jN_j}(f^{m+\sum_{l=1}^{j-1}n_lN_l}y),\sum_{i=1}^{N_j}\frac{1}{N_j}\mathcal E_{n_j}(p_j))+\delta_j/2+d(\alpha_j,\alpha_{j+1}).
	\end{split}
\end{equation}
and 
\begin{equation}\label{G_K-equation-4}
	\begin{split}
		d(\mathcal E_{n_{j+1}\mathfrak{N}}(f^{m+\sum_{l=1}^{j}n_lN_l}(y)),\alpha_{j+1})&\leq d(\mathcal E_{n_{j+1}\mathfrak{N}}(f^{m+\sum_{l=1}^{j}n_lN_l}(y)),\sum_{i=1}^{\mathfrak{N}}\frac{1}{\mathfrak{N}}\mathcal E_{n_{j+1}}(p_{j+1}))\\
		&\ \ \ \ +d(\sum_{i=1}^{\mathfrak{N}}\frac{1}{\mathfrak{N}}\mathcal E_{n_{j+1}}(p_{j+1}),\alpha_{j+1})\\
		&\leq d(\mathcal E_{n_{j+1}\mathfrak{N}}(f^{m+\sum_{l=1}^{j}n_lN_l}(y)),\sum_{i=1}^{\mathfrak{N}}\frac{1}{\mathfrak{N}}\mathcal E_{n_{j+1}}(p_{j+1}))+\delta_{j+1}/2.
	\end{split}
\end{equation}
Note that $y$ limit-shadows some $\mathfrak{C}(\xi)$, so by Lemma \ref{measure distance}
\begin{equation}\label{G_K-equation-3}
\lim_{j\to\infty}d(\mathcal E_{n_jN_j}(f^{m+\sum_{l=1}^{j-1}n_lN_l}(y)),\sum_{i=1}^{N_j}\frac{1}{N_j}\mathcal E_{n_j}(p_j))=0,
\end{equation}
and 
\begin{equation}\label{G_K-equation-5}
	\lim_{j\to\infty}d(\mathcal E_{n_{j+1}\mathfrak{N}}(f^{m+\sum_{l=1}^{j}n_lN_l}(y)),\sum_{i=1}^{\mathfrak{N}}\frac{1}{\mathfrak{N}}\mathcal E_{n_{j+1}}(p_{j+1}))=0.
\end{equation}
(\ref{G_K-equation-1})-(\ref{G_K-equation-5}) and (\ref{equation-A}) result in
$$\lim_{k\to\infty}d(\mathcal E_{M_k}(y),\alpha_{M_k}^{\prime})=0.$$
Thus we have proved $S\subseteq G_K$.

(2): For any $y=S_{\mathfrak{C}(\xi)}$, $y$ limit-shadows some $\mathfrak{C}(\xi)$. Thus $\omega_f(y)=\omega(\mathfrak{C}(\xi))$. Note that $\lim_{j\to\infty}d_H(\mathfrak{C}_j, \Lambda)=0$, so $\omega_f(y)=\Lambda$.

(3): For distinct $x,y\in S$, we can assume that $x=S_{\mathfrak{C}(\xi)},$ and $y=S_{\mathfrak{C}(\eta)},$ where $\mathfrak{C}(\xi)=\langle u_1,u_2,\cdots\rangle,$
$\mathfrak{C}(\eta)=\langle z_1,z_2,\cdots\rangle$ and $\xi\neq\eta$. Note that $x,y$ limit-shadows $\mathfrak{C}(\xi), \mathfrak{C}(\eta)$ respectively. So for any $t>0$, there exists a $k_0\in\mathbb{N}$ such that for any $n\geq m+\sum_{j=1}^{i_{k_0}}n_jN_j$
\begin{equation}\label{pesudo-equal-1}
d(f^{n-1}(x),u_n)<t/2,\ \ \ d(f^{n-1}(y),z_n)<t/2.
\end{equation}
\noindent Note that (\ref{pesudo-construction}), then for any $k\in\mathbb{N}$ and any $s\in[\sum_{j=1}^{i_{k}}n_jN_j+1,\sum_{j=1}^{i_{k}+1}n_jN_j]$,
\begin{equation}\label{pesudo-equal-2}
u_{m+s}=z_{m+s}.
\end{equation}
By (\ref{pesudo-equal-1}) and (\ref{pesudo-equal-2}), we have for any $k\geq k_0$ and any $s\in[\sum_{j=1}^{i_{k}}n_jN_j+1,\sum_{j=1}^{i_{k}+1}n_jN_j]$,
\begin{equation}\label{eq-1}
d(f^{m+s-1}(x),f^{m+s-1}(y))<t.
\end{equation}
Combining (\ref{eq-1}) and (\ref{PS-N-2}),
\begin{align*}
&\limsup_{n\to \infty}\frac{1}{n}|\{j\in [0,n-1]:\ d(f^j(x),f^j(y))<t\}|\\
\ge & \limsup_{k\to \infty}\frac{1}{m+\sum_{j=1}^{i_{k}+1}n_jN_j}|\{j\in [0,m+\sum_{j=1}^{i_{k}+1}n_jN_j-1]:\ d(f^jx,f^jy)<t\}|\\
\ge & \limsup_{k\to \infty}\frac{n_{i_{k}+1}N_{i_{k}+1}}{m+\sum_{j=1}^{i_{k}+1}n_jN_j}\\
= & 1.
\end{align*}

On the other hand, there exists a $h\in\mathbb{N}$ such that $\xi_h\neq \eta_h$ since $\xi\neq\eta$. By the definition of $[k]$, there is a strictly increasing integer sequence $\{k_j\}_{j=1}^{\infty}$ such that $[k_j]=h$ for any $j\in\mathbb{N}$, which implies
\begin{equation}\label{pesudo-distance-1}
p_{i_{k_j}}^{\xi_{[k_j]}}\neq p_{i_{k_j}}^{\eta_{[k_j]}}.
\end{equation}
By the item \textbf{(c)} of Lemma \ref{lemma-double} and (\ref{pesudo-distance-1}), there is a $j_0^*\in\mathbb{N}$ such that for any $j\geq j_0^*$,
\begin{equation}\label{pesudo-distance-3}
\frac{|\ \{i\in[\sum_{i=1}^{i_{k_{j}}-1}n_iN_i+1,\sum_{i=1}^{i_{k_{j}}}n_iN_i]\ |\ d(u_{m+i},z_{m+i})<\zeta-\delta_{i_{k_{j}}}/2\}\ |}{n_{i_{k_{j}}}N_{i_{k_{j}}}}<\tau_{i_{k_{j}}}.
\end{equation}
Note that $x,y$ limit-shadows $\mathfrak{C}(\xi), \mathfrak{C}(\eta)$ respectively. So there exists a $j_0\in\mathbb{N}$ such that for any $n\geq m+\sum_{i=1}^{i_{k_{j_0}}-1}n_iN_i$
\begin{equation}\label{pesudo-distance-2}
d(f^{n-1}(x),u_n)<\zeta/8,\ \ \ d(f^{n-1}(y),z_n)<\zeta/8.
\end{equation}
Let $\hat{j}$ large enough such that $\delta_{i_{k_{\hat{j}}}}\leq \zeta/2$ and $\hat{j}\geq\mathrm{max}\{j_0,j_0^*\}$. Combining (\ref{pesudo-distance-3}) and (\ref{pesudo-distance-2}),

\begin{align*}
&\liminf_{n\to \infty}\frac{1}{n}|\{j\in [0,n-1]:\ d(f^i(x),f^i(y))<\zeta/2\}|\\
\leq & \liminf_{j\to \infty}\frac{1}{m+\sum_{i=1}^{i_{k_j}}n_iN_i}|\{i\in [0,m+\sum_{j=1}^{i_{k_j}}n_iN_i-1]:\ d(f^i(x),f^i(y))<\zeta/2\}|\\
\leq & \liminf_{ j\to \infty}\frac{1}{m+\sum_{i=1}^{i_{k_j}}n_iN_i}\{m+\sum_{j=1}^{i_{k_j}-1}n_iN_i+|\{i\in [m+\sum_{j=1}^{i_{k_j}-1}n_iN_i,m+\sum_{j=1}^{i_{k_j}}n_iN_i-1]:\\
& \ d(f^i(x),f^i(y))<\zeta/2\}|\}\\
\leq & \liminf_{j\to \infty}\frac{m+\sum_{j=1}^{i_{k_j}-1}n_iN_i}{m+\sum_{i=1}^{i_{k_j}}n_iN_i}+\tau_{i_{k_{j}}}\\
= & 0.
\end{align*}
So $x,y$ is a DC1-scrambled pair.

(4): Implied by (3).

\textbf{(b):} Suppose that $(X,f)$ has exponential shadowing property with exponent $\lambda$. Given $\alpha\in\mathcal{A}$, we will construct an uncountable $\alpha$-DC1-scrambled set $S^\alpha$. The construction is similar to the construction in the proof of item \textbf{(a)}. We put the differences in the following

(1) By  Lemma \ref{exponential-s-limit-shadowing} there is a $\delta_i\in(0,\eps_i)$ such that 
any $\delta_i$-limit-pseudo-orbit can be both  $(\varepsilon_i,\lambda)$-shadowed and limit-shadowed by some point in $X,$ and $\lim_{i\to\infty}\delta_i=0.$

(2) $N_i$ chosed here should be even number, satisfies (\ref{PS-N-1}), (\ref{PS-N-2}) and 
\begin{equation}\label{PS-N-add}
	\frac{n_{k+1}N_{k+1}-c_k}{m+\sum_{i=1}^{k+1}n_iN_i}>1-\tau_k\text{ for any $k\geq1$},
\end{equation}
where 
\begin{equation}\label{PS-N-add-2}
c_k:=\min\{c\in\mathbb{N}^+\mid\alpha(m+\sum_{j=1}^{k}n_jN_j+c)\tau_k>m+\sum_{j=1}^{k}n_jN_j+\frac{4\varepsilon_1}{1-e^{-\lambda}}\}\text{ for any $k\geq1$}.
\end{equation}
(\ref{PS-N-add-2}) can be satisfied since $\lim\limits_{n\to\infty}\alpha(n)=\infty$.

(3) $\mathfrak{C}(\xi)$ also can be seen as $$\mathfrak{C}(\xi)=orb(p_0,m)\underbrace{\mathfrak{C}_1\cdots\mathfrak{C}_1}_{N_1}\underbrace{\mathfrak{C}_2\cdots\mathfrak{C}_2}_{N_2}\underbrace{\mathfrak{C}_3\cdots\mathfrak{C}_3}_{N_3}\cdots,$$
it is a {$\delta_1$-limit-pseudo-orbit}. Then by  Lemma \ref{exponential-s-limit-shadowing}, $\mathfrak{C}(\xi)$ is both  $(\varepsilon_1,\lambda)$-shadowed and limit-shadowed by some point $S^\alpha_{\mathfrak{C}(\xi)}$. Denote  $$S^\alpha:=\bigcup_{\xi\in\{1,2\}^{\infty}}S^\alpha_{\mathfrak{C}(\xi)}.$$
We complete this proof by proving the following four facts:
\begin{description}
	\item[(i)] $S^{\alpha}\subseteq G_K$;
	\item[(ii)] For any $y\in S^\alpha$, $\omega_f(y)=\Lambda$;
	\item[(iii)] For any distinct $x,y\in S^\alpha$, $x,y$ is a $\alpha$-DC1-scrambled pair;
	\item[(iv)] $S^\alpha_{\mathfrak{C}(\xi)}\neq S^\alpha_{\mathfrak{C}(\eta)}$ if $\xi\neq \eta$, which implies $S^\alpha$ is uncountable.
\end{description}
Item (iv) is directly from item (iii), the proofs of item (i) and item (ii) are the same as item (1) and item (2) in the proof of item \textbf{(a)}. We only need to prove the different part in item (iii).

 For distinct $x,y\in S^\alpha$, we can assume that $x=S^\alpha_{\mathfrak{C}(\xi)}, y=S^\alpha_{\mathfrak{C}(\eta)},$ where $\mathfrak{C}(\xi)=\langle u_1,u_2,\cdots\rangle,$
$\mathfrak{C}(\eta)=\langle z_1,z_2,\cdots\rangle$ and $\xi\neq\eta$. Recall that $\{\tau_i\}_{i=1}^{\infty}$ is a strictly decreasing sequence with $\lim_{i\to\infty}\tau_i=0$. So for any $t>0$, there exists a $k_0\in\mathbb{N}^+$ such that for any $n\geq k_0$, we have $\tau_n<t$. Denote $I_k:=[\sum_{j=1}^{i_{k}}n_jN_j+1,\sum_{j=1}^{i_{k}+1}n_jN_j]$, note that (\ref{pesudo-construction}), then for any $k\in\mathbb{N}$ and any $s\in I_k$,
\begin{equation}\label{pesudo-equal-2-2}
	u_{m+s}=z_{m+s}.
\end{equation}
By (\ref{pesudo-equal-2-2}) and the definitions of $x$ and $y$, we have
\begin{equation}\label{eq-1-2}
	\sum_{s\in I_k}d(f^{m+s-1}(x),f^{m+s-1}(y))<2\cdot2\varepsilon_1\sum_{r=0}^{\frac{n_{i_k+1}N_{i_k+1}}{2}-1}e^{-r\lambda}<\frac{4\varepsilon_1}{1-e^{-\lambda}}.
\end{equation}
Combining (\ref{eq-1-2}) and (\ref{PS-N-add-2}), for any $k\geq k_0$ and $c\in I_k$, we have
\begin{equation}\label{eq-1-3}
	\sum_{j=0}^{m+c-1}d(f^j(x),f^j(y))\leq(m+\sum_{i=1}^{i_k}n_iN_i)\diam X+\frac{4\varepsilon_1}{1-e^{-\lambda}}<\alpha(m+\sum_{j=1}^{i_k}n_jN_j+c_{i_k})t.
\end{equation}
As a result, we have
\begin{align*}
	&\limsup_{n\to \infty}\frac{1}{n}|\{1\leq i \leq n:\sum_{j=0}^{i-1}d(f^{j}(x),f^{j}(y))<\alpha(i)t\}|\\
	\ge & \limsup_{ k\to \infty}\frac{1}{m+\sum_{j=1}^{i_{k}+1}n_jN_j}|\{i\in [m+\sum_{j=1}^{i_k}n_jN_j+c_{i_k},m+\sum_{j=1}^{i_{k}+1}n_jN_j]:\sum_{j=0}^{i-1}d(f^{j}(x),f^{j}(y))<\alpha(i)t\}|\\
	\ge & \limsup_{k\to \infty}\frac{n_{i_{k}+1}N_{i_{k}+1}-c_{i_k}}{m+\sum_{j=1}^{i_{k}+1}n_jN_j}\\
	= & 1.
\end{align*}
The rest of the proof of item (iii) is the same as item (3) in the proof of item \textbf{(a)}.\qed

\subsection{Strong-elementary-dense Property}
For any $m\in\N$ and $\{\nu_i\}_{i=1}^m \subseteq M(X)$,   we write $\mathrm{cov}\{\nu_i\}_{i=1}^m$ for the convex combination of $\{\nu_i\}_{i=1}^m$,   namely,
$$\mathrm{cov}\{\nu_i\}_{i=1}^m=\mathrm{cov}(\nu_1,\cdots,\nu_m):=\left\{\sum_{i=1}^mt_i\nu_i:t_i\in[0,  1],  1\leq i\leq m~\textrm{and}~\sum_{i=1}^mt_i=1\right\}.$$
\begin{Def}
We say that $(X,  f)$ satisfies the strong-elementary-dense property if for any $K=\mathrm{cov}\{\mu_i\}_{i=1}^m\subseteq M_f(X)$ and any $\eps>0$,  there exist compact invariant subset $\Lambda_i\subsetneq\Lambda\subsetneq X$,   $1\leq i\leq m$ such that
\begin{description}
  \item [(a)] $(\Lambda,f)$ has the strong specification property, more generally $\Lambda\in ICM.$
  \item [(b)] $d_H(K,  M_f(\Lambda))<\eps$,   $d_H(\mu_i,  M_f(\Lambda_i))<\eps$.
  \item [(c)]  There is no fixed point in $\Lambda$.
\end{description}

\end{Def}

\begin{mainlemma}\label{strong-elementary-entropy-dense}
Suppose $(X,  f)$ is positively expansive, mixing and has the shadowing property.   Then $(X,  f)$ satisfies the strong-elementary-dense property.
\end{mainlemma}

\begin{proof}
 Actually, it is enough to finish the proof of our main theorems if this lemma holds for any $K=\mathrm{cov}\{\mu_1,\mu_2\}\subseteq M_f(X)$. So, for brevity, we just prove this lemma for such $K$. Let $e$ be the expansive constant and $\eps^*=\mathrm{min}\{\eps/3,e/3\}$. By shadowing property, there is a $\delta\in(0,\eps^*)$ such that any $\delta$-pseudo-orbit is {$\eps^*$-shadowed} by some point in $X$. Note $X\in ICM$ since $(X,f)$ is mixing. So, by Lemma \ref{lemma-single} for a fix $x\in X$, there is an $n\in\mathbb{N}$ and $p_1,p_2\in B(x,\delta/2)$ such that $p_1\in Per_n,p_2\in Per_{n+1}$ with $\mathcal E_{n}(p_1)\in B(\mu_1,\eps^*),\mathcal E_{n+1}(p_2)\in B(\mu_2,\eps^*)$  Denote $I=\{p_1,f(p_1),\cdots,f^{n-1}(p_1),p_2,f(p_2),\cdots,f^n(p_2)\}.$
Let $\Sigma\subseteq \Sigma_{I}$ be a subshift of finite type with transition matrix $L$:

\begin{equation}\label{Eq:matrix}
\bordermatrix{%
                & p_1     & f(p_1)   &\cdots    &           & p_2     &           &\cdots              & f^n(p_2)\cr
p_1             & 0       & 1      & 0        &           & \cdots  &           &\cdots              & 0\cr
f(p_1)            & 0       & 0      & 1        & 0         & \cdots  &           &\cdots              & 0\cr
\vdots          & \vdots  &        &          &\ddots     &         &           &                    & \vdots\cr
f^{n-1}(p_1)      & 1       & 0      &\cdots    & 0         & 1       & 0         &\cdots              & 0\cr
p_2             & 0       & 0      &\cdots    & \cdots    & 0       & 1         & \cdots             & 0\cr
\vdots          &\vdots   &        & \ddots   &           & \vdots  & \vdots    &\ddots              & \vdots\cr
f^{n-1}(p_2 )     & 0       &        &          & 0         & 0       &           &                    & 1\cr
f^n(p_2)          & 1       & 0      &\cdots    & 0         & 1       & 0         &\cdots              & 0\cr
},
\end{equation}

\noindent where $\Sigma_{I}:=\{\ \langle\theta_1,\theta_2,\cdots\rangle\ :\ \theta_i\in I, i\in \mathbb{N}^+\}$ is a full shift with alphabet $I$. Actually we define a subshift whose points is the infinite connection of the orbits of $p_1$ and $p_2$ such like $\langle f^{n-2}(p_1),f^{n-1}(p_1),p_1,f(p_1),\cdots,f^{n-1}(p_1),p_2,f(p_2),\cdots,f^{n-1}(p_2)\cdots \text{repeating} \cdots\rangle$. And this is what the transition matrix exactly implies. The directed graph deduced by $L$ is as follows:

\begin{equation*}
\includegraphics[height=2.5cm]{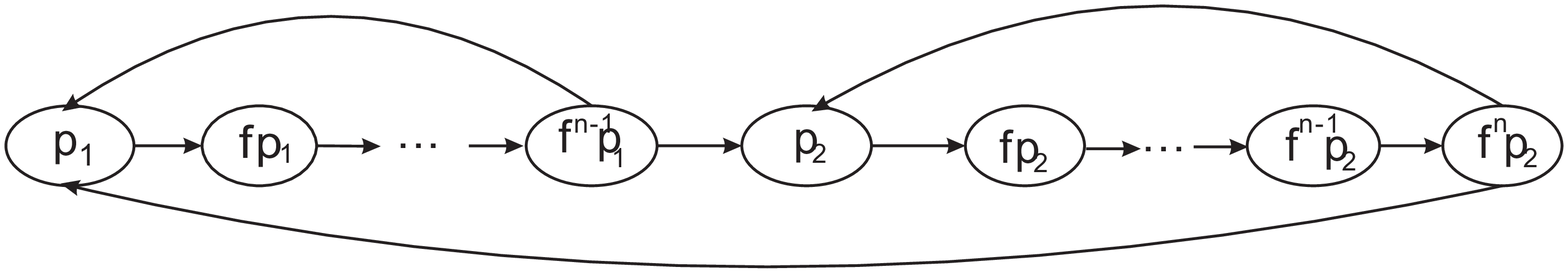}
\end{equation*}

\text{ }\\\text{ }\\
\noindent Obviously, it is a strongly connected graph, which implies $L$ is irreducible. By \cite[Proposition 17.9]{Sig}, $(\Sigma,\sigma)$ is transitive. For any $k\geq n^2$, we can write $k$ as $k=sn+t$ with $s\geq n$ and $t\le n$. Thus $k=(s-t)n+t(n+1)$. In other words, for any $k\geq n^2$, there is a block with length $k$ as
$$\langle\underbrace{orb(p_1,n)\cdots orb(p_1,n)}_{s-t}\underbrace{orb(p_2,n+1)\cdots orb(p_2,n+1)}_{t}\rangle.$$
\noindent Then for any $a,b\in I$, there exists an $m_0(\ge n^2)$ such that for any $m\geq m_0$, $\{x\in\Sigma\ |\ x_1=a, x_m=b\}\neq\emptyset$. So $(\Sigma,\sigma)$ is mixing by \cite[Proposition 17.8, 17.10 (3)]{Sig}. Further, $(\Sigma,\sigma)$ has strong specification property by Lemma \ref{mixing+shadowing=spe}. For any $\theta\in \Sigma$, one can observe that it is a $\delta$-pseudo-orbit. Denote
$$Y_\theta:=\{x\in X\ |\ d(f^{i-1}(x),\theta_i)\leq e/3,\ i\geq 1\}.$$
\noindent Then $Y_\theta$ is nonempty by shadowing property and singleton by expansiveness. So $f(Y_\theta)=Y_{\sigma(\theta)}$.
Define
$\Lambda=\bigcup_{\theta\in\Sigma}Y_\theta,\   \text{and}$
\begin{align*}
\pi:\ &\Sigma\rightarrow\Lambda\\
&\theta\mapsto Y_\theta.
\end{align*}
One can check that $\pi$ is continuous. Thus $\Lambda$ is closed and $(\Lambda,f)$ is a factor of $(\Sigma,\sigma)$. By \cite[Proposition 21.4 (c)]{Sig}, $(\Lambda,f)$ has strong  specification property.

Let $\Lambda_i=orb(p_i),i=1,2$. Then $\Lambda_i\subsetneq \Lambda$ and $d_H(\mu_i,\mathcal M_f(\Lambda_i))<\eps^*<\eps$. For any $\tilde{\mu}\in K$, let $\mu\in B(\tilde{\mu},\eps^*)$ with $\mu=\frac{c_1}{c}\mu_1+\frac{c_2}{c}\mu_2$, where $c_1,c_2\in\mathbb{N}^+, c=c_1+c_2$. Let $\mathfrak{C}_1=orb(p_1,n),\mathfrak{C}_2=orb(p_2,n+1)$,
$$\mathfrak{C}_\mu=\underbrace{\mathfrak{C}_1\cdots\mathfrak{C}_1}_{(n+1)c_1}\underbrace{\mathfrak{C}_2\cdots\mathfrak{C}_2}_{nc_2},$$
and $\theta=\mathfrak{C}_\mu\mathfrak{C}_\mu\cdots$. Then $\theta\in\Sigma$ and $Y_\theta\in Per_{n(n+1)c}$, which implies $V_f(Y_\theta)=\mathcal E_{n(n+1)c}(Y_\theta)$.
\begin{align*}
d(\mathcal E_{n(n+1)c}(Y_\theta),\mu) &\leq d(\mathcal E_{n(n+1)c}(Y_\theta),\delta_{\mathfrak{C}_\mu})+d(\delta_{\mathfrak{C}_\mu},\frac{c_1}{c}\mu_1+\frac{c_2}{c}\mu_2)\\
&\leq \eps^*+d(\frac{n(n+1)c_1}{n(n+1)c}\delta_{\mathfrak{C}_1}+\frac{n(n+1)c_2}{n(n+1)c}\delta_{\mathfrak{C}_2},\frac{c_1}{c}\mu_1+\frac{c_2}{c}\mu_2)\\
&\leq \eps^*+\frac{c_1}{c}\eps^*+\frac{c_2}{c}\eps^*\\
&< 2\eps^*.
\end{align*}
Thus $d(V_f(Y_\theta),\tilde{\mu})<3\eps^*\leq\eps$. Further, $K\subseteq B(\mathcal M_f(\Lambda),\eps)$. One the other hand, $\mathcal M_f(\Lambda)\subseteq B(K,\eps)$ by its definition. So $d_H(K,\mathcal M_f(\Lambda))<\eps$.

Finally, from the proof of Lemma \ref{lemma-single},  we can assume that $p_1$ and $p_2$ are not fixed points, so there is no fixed point in $\Lambda$.
\end{proof}

\section{Proof of Main Theorems}\label{sectionEndOfProofs}

\subsection{Some Lemmas}
For any $x\in X$,   we define the measure center of $x$ as
$$C^*_x:=\overline{\bigcup_{\mu\in \mathcal M_f(\omega_f(x))}S_{\mu}}.  $$
Furthermore,   we define the measure center of a closed invariant set $\Lambda\subseteq X$ as
$$C^*_\Lambda:=\overline{\bigcup_{\mu\in \mathcal M_f(\Lambda)}S_{\mu}}.  $$

\begin{Lem}\label{G_K-subset-omega}
For $(X,  f)$,   let $\Lambda\subsetneq X$ be closed $f$-invariant and $K\subseteq \mathcal M_f(\Lambda)$ be a nonempty compact connected set.
\begin{description}
  \item[(1)] Suppose $\bigcap_{\mu\in K}S_{\mu}=C^*_\Lambda$.   Then
  $$G_K^{\Lambda}\subseteq\{x\in X:\bigcap_{\mu\in K}S_{\mu}=\omega_{\underline{d}}(x)= \omega_{\overline{d}}(x)= \omega_{B^*}(x)= C^*_\Lambda\subseteq \omega_f(x)=\Lambda\}.  $$
  \item[(2)] Suppose $\bigcap_{\mu\in K}S_{\mu}=\overline{\bigcup_{\nu\in K}S_{\nu}}\subsetneq C^*_\Lambda$.   Then
      $$G_K^{\Lambda}\subseteq\{x\in X:\bigcap_{\mu\in K}S_{\mu}=\omega_{\underline{d}}(x)= \omega_{\overline{d}}(x)\subsetneq  \omega_{B^*}(x)= C^*_\Lambda\subseteq \omega_f(x)=\Lambda\}.  $$
  \item[(3)] Suppose $\bigcap_{\mu\in K}S_{\mu}\subsetneq\overline{\bigcup_{\nu\in K}S_{\nu}}=C^*_\Lambda$.   Then
     $$G_K^{\Lambda}\subseteq\{x\in X:\bigcap_{\mu\in K}S_{\mu}=\omega_{\underline{d}}(x)\subsetneq \omega_{\overline{d}}(x)=  \omega_{B^*}(x)= C^*_\Lambda\subseteq \omega_f(x)=\Lambda\}.  $$
  \item[(4)] Suppose $\bigcap_{\mu\in K}S_{\mu}\subsetneq\overline{\bigcup_{\nu\in K}S_{\nu}}\subsetneq C^*_\Lambda$.   Then
      $$G_K^{\Lambda}\subseteq\{x\in X:\bigcap_{\mu\in K}S_{\mu}=\omega_{\underline{d}}(x)\subsetneq \omega_{\overline{d}}(x)\subsetneq  \omega_{B^*}(x)= C^*_\Lambda\subseteq \omega_f(x)=\Lambda\}.  $$
\end{description}
\end{Lem}
\begin{proof}
For any $x\in G_K^{\Lambda}$,   $\omega_f(x)=\Lambda$ by definition.   So by item \textbf{(4)} of Proposition \ref{density-measure-prop},   $$\omega_{B^*}(x)=C^*_x=\overline{\bigcup_{\mu\in \mathcal M_f( \omega_f(x))}S_{\mu}}=\overline{\bigcup_{\mu\in \mathcal M_f( \Lambda)}S_{\mu}}=C^*_\Lambda.  $$
Consequently,   one uses (\ref{omega-relation}) and obtains that
\begin{equation}\label{G_K-omega-equation-1}
  \omega_{\underline{d}}(x)\subseteq \omega_{\overline{d}}(x)\subseteq \omega_{B^*}(x)=C^*_{\Lambda}\subseteq\omega_f(x)=\Lambda.
\end{equation}
Note that $V_f(x)=K $, then
\begin{equation}\label{G_K-omega-equation-2}
  \omega_{\underline{d}}(x)=\bigcap_{\mu\in V_f(x)}S_{\mu}=\bigcap_{\mu\in K}S_{\mu}~\textrm{and}~\omega_{\overline{d}}(x)=\overline{\bigcup_{\mu\in V_f(x)}S_{\mu}}=\overline{\bigcup_{\nu\in K}S_{\nu}}
\end{equation}
by Proposition \ref{density-measure-prop}.
Therefore,   a convenient use of \eqref{G_K-omega-equation-1} and \eqref{G_K-omega-equation-2} yields (1)-(4).
\end{proof}

Given an invariant compact subset $\Lambda\subset X$ and a continuous function $\varphi:X\rightarrow \mathbb{R},$
we denote
$$L^\Lambda_\varphi=\left[\inf_{\mu\in \mathcal M_f( \Lambda)}\int\varphi d\mu,  \,  \sup_{\mu\in \mathcal M_f( \Lambda)}\int\varphi d\mu\right].$$

\begin{Lem}\label{lambda-theta-level}
Suppose $(X,  f)$ is positively expansive, mixing and has the shadowing property.   Let $\varphi:X\rightarrow \mathbb{R}$ be a continuous function and assume that $Int(L_\varphi)\neq \emptyset$.  Then for any $a\in Int(L_\varphi)$, there are two $f$-invariant compact subsets $\Lambda\subsetneq \Theta \subsetneq X$ such that
 \begin{description}
   \item[(1)] $\Lambda$ has the strong specification property, more generally $\Lambda\in ICM$ with
   \begin{equation}
  C^*_\Lambda=\Lambda,  ~a\in Int(L_\varphi^{\Lambda}).
\end{equation}
   \item[(2)] $\Theta$ is internally chain mixing with
   \begin{equation}
  \Lambda=C^*_{\Theta}\subsetneq\Theta,  ~a\in Int(L^\Theta_\varphi).
\end{equation}
   \item[(3)] There is no fixed point in $\Lambda$ and $\Theta$.
 \end{description}
In particular, $Int(L^\Lambda_\varphi)=Int(L^\Theta_\varphi)\neq \emptyset.$
\end{Lem}
\begin{proof}
$\mathbf{(1)}$: For any $a\in Int(L_\varphi)$, we can choose $\mu, \nu\in\mathcal M_f(X)$ with $\int\varphi d\mu<a<\int\varphi d\nu$. Denote
$$\zeta=\mathrm{min}\{a-\int\varphi d\mu,\ \int\varphi d\nu-a\}.$$
By Lemma \ref{strong-elementary-entropy-dense}, we obtain a $\Lambda\subsetneq X$ and $\Lambda$ has strong specification property with two measure $\widetilde{\mu},\widetilde{\nu}\in\mathcal M_f(\Lambda)$ such that
$$\left|\int\varphi d\widetilde{\mu}-\int\varphi d\mu\right|<\zeta/2,  ~\left|\int\varphi d\widetilde{\omega}-\int\varphi d\omega\right|<\zeta/2.$$
Then $\int\varphi d\widetilde{\mu}<a<\int\varphi d\widetilde{\nu}$, which implies $a\in Int(L_\varphi^{\Lambda})$. By Lemma \ref{lemma-full}, there is a $\rho\in\mathcal M_f(\Lambda)$ with full support. Thus $C^*_\Lambda=\Lambda$.

$\mathbf{(2)}$: By \cite[Corollary 4.21]{DT}, we can find a $z\notin\Lambda$ such that $\omega_f(z)=\Lambda$. Let $A=\orb(z,  f)\cup \Lambda$.   Note that $A$ is closed and $f$-invariant.   So by \cite[Lemma 7.10]{DT},  there is a point $x\in X$ such that
$$A\subseteq\omega_f(x)\subseteq \cup_{l=0}^{\infty}f^{-l}A.  $$
In particular,
$$\overline{\bigcup_{y\in\omega_f(x)}\omega_f(y)}\subseteq A\subseteq\omega_f(x)\neq X.  $$
Let $\Theta=\omega_f(x)$.  Then $\Theta\in ICT$ by Lemma \ref{omega-subset-ICT}. Note that $\Lambda\subsetneq\Theta$ and $\Lambda\in ICM$, so $\Theta\in ICM$ by Lemma \ref{ICT-in-ICM}. By \cite[Lemma 6.6]{DT},
$$\mathcal M_f(\Theta)=\mathcal M_f(A)=\mathcal M_f(\Lambda).  $$
This implies that $C_\Theta^*=\Lambda\subsetneq\Theta$ and $Int(L_\varphi^\Theta)=Int(L_\varphi^\Lambda)\neq \emptyset$.

$\mathbf{(3)}$: Finally, by Lemma \ref{strong-elementary-entropy-dense}, there is no fixed point in $\Lambda.$ By $\mathcal M_f(\Theta)=\mathcal M_f(\Lambda),$ there is also no fixed point in $\Theta$.
\end{proof}

\begin{Def}\label{definition of entropy-dense property}
	We say $(X,f)$ satisfies the entropy-dense property if for any $\mu \in \mathcal{M}_f(X)$, for any neighbourhood $G$ of $\mu$ in $\mathcal{M}(X)$, and for any $\eta >0$, there exists a closed $f$-invariant set $\Lambda_{\mu}\subseteq X$, such that  $\mathcal{M}_f(\Lambda_{\mu})\subseteq G$ and $\htop(f,\Lambda_{\mu})>h_{\mu}(f)-\eta$. By the classical variational principle, it is equivalent that for any neighbourhood $G$ of $\mu$ in $\mathcal{M}(X)$, and for any $\eta >0$, there exists a $\nu \in \mathcal{M}_{f}^e(X)$ such that $h_{\nu}(f)>h_{\mu}(f)-\eta$ and $\mathcal{M}_f(S_{\nu})\subseteq G$.
\end{Def}
By \cite[Proposition 2.3(1)]{PS}, entropy-dense property holds for systems with approximate product property. And approximate product property is  weaker than specification property from their definitions. So we have the following.

\begin{Prop}\cite[Proposition 2.3 (1)]{PS}\label{proposition of entropy-dense property}
	Suppose that $(X,f)$ is a dynamical system satisfying  specification property.
	Then $(X,f)$ has entropy-dense property.
\end{Prop}
\begin{Prop}\label{prop-IrregularMultianalysis-22222}
	Suppose that $(X,f)$ is a dynamical system satisfying  specification property. Then for any  $n\in\N$,  there exist $f$-invariant ergodic measures $\{\omega_k\}_{k=1}^n$  such that $\{S_{\omega_k}\}_{k=1}^n$ are pairwise disjoint and $\cup_{k=1}^{n}S_{\omega_k}\neq X$.
\end{Prop}
\begin{proof}
	Since $(X,f)$ has specification property and $(X,d)$ be a nondegenerate $($i.e,
	with at least two points$)$ compact metric space,   then $(X,f)$  is not uniquely ergodic. Thus there exists $f$-invariant measures $\mu\neq \nu$ on $X$ such that $\nu\neq \omega.$ Choose $0<\theta_1<\theta_2<\cdots<\theta_{n+1}<1.$ 
	Let $\nu_k:=\theta_k\mu+(1-\theta_k)\nu$. Then $\nu_i\neq \nu_j$ for any $1\leq i<j\leq n+1.$
	Denote $\rho_0=\min\{\rho(\nu_i,\nu_j):1\leq i<j\leq n+1\}.$
	By Proposition \ref{proposition of entropy-dense property}, for each $k=1,  \cdots,  n+1$,   then there exists  $\omega_k\in \mathcal M_f^e(X)$ with $\mathcal M_f(S_{\omega_k})\subseteq \mathcal{B}(\nu_k,  \frac{1}{3}\rho_0).$
	Note that $\{\mathcal{B}(\nu_k,  \frac{1}{3}\rho_0)\}_{k=1}^{n+1}$ are pairwise disjoint.   So 
	$\{\mathcal M_f(S_{\omega_k})\}_{k=1}^{n+1}$ are pairwise disjoint.   As a result,
	$\{S_{\omega_k}\}_{k=1}^{n+1}$ are pairwise disjoint and $\cup_{k=1}^{n}S_{\omega_k}\neq X$.
\end{proof}

\begin{Prop}\label{proposition of distal}
	Suppose that $(X,f)$ is a dynamical system. If there is no fixed point in $X,$
	then $G_\mu$ has distal pair for any $\mu\in\mathcal M_f^e(X)$.
\end{Prop}
\begin{proof}
	$G_\mu\neq \emptyset$ since $\mu\in\mathcal{M}_f^e(X)$. Let $x\in  G_\mu$, then $f(x)\in G_\mu$. Assume that $x, f(x)$ is not distal, then $\omega_f(x)$ contains a fixed point.
\end{proof}

It is easy to check Theorem \ref{thm-NR-new} and \ref{thm-NR} can be deduced from Theorem \ref{thm-Ir} or \ref{thm-Level} so we only need to show Theorem \ref{thm-Ir} or \ref{thm-Level}.

\subsection{Proof of Theorem \ref{thm-Ir}}
Now we state a abstract result on strongly distributional chaos of $\{x\in X\mid x \text{ satisfies case } (i)\}\cap NR(f)\cap I_{\varphi}(f).$ Combining with Theorem \ref{thm-2} and Theorem \ref{thm-3}, we have Theorem \ref{thm-Ir}.
\begin{Thm}\label{thm-Is}
	Suppose $(X,  f)$ is positively expansive, mixing and has the exponential shadowing property.  Let $\varphi$ be a continuous function on $X$. If $I_{\varphi}(f)\neq\emptyset$, then   $\{x\in X\mid x \text{ satisfies case } (i)\}\cap NR(f)\cap I_{\varphi}(f)$ is strongly distributional chaotic for any $i\in\{1,2,3,4,5,6,1',2',3',4',5',6'\}.$
\end{Thm}
\begin{proof}
	$I_{\varphi}(f)\neq\emptyset$ implies $Int(L_\varphi)\neq\emptyset$. Then there are $\Lambda\subsetneq \Theta \subsetneq X$ satisfying Lemma \ref{lambda-theta-level}. By Proposition \ref{prop-IrregularMultianalysis-22222}, Lemma \ref{lemma-full} and Lemma \ref{lemma-ergo}, we can take $\mu_1,\mu_2\in\mathcal M_f^e(X)$ with disjoint support, and $\mu,\nu\in\mathcal M_f^e(\Lambda)$ with full support satisfying
	$$\int\varphi d\mu_1,\ \int\varphi d\mu_2,\ \int\varphi d\nu,\ \text{and}\ \int\varphi d\mu\ \text{are\ mutually\ unequal}.$$
	Also, $S_{\mu_1}\cup S_{\mu_2}\neq \Lambda$ by Proposition \ref{prop-IrregularMultianalysis-22222}. Let $0<\theta_1<\theta_2<1$ and
	\begin{align*}
		K_1\ &:=\ \ \mathrm{cov}\{\theta_1\mu_1+(1-\theta_1)\mu_2,\theta_2\mu_1+(1-\theta_2)\mu_2\},\\
		K_2\ &:=\ \ \mathrm{cov}\{\mu_1,\mu\}\cup\mathrm{cov}\{\mu_2,\mu\},\\
		K_3\ &:=\ \ \mathrm{cov}\{\mu_1,\theta_1\mu_1+(1-\theta_1)\mu\},\\
		K_4\ &:=\ \ \mathrm{cov}\{\mu_1,\mu_2\},\\
		K_5\ &:=\ \ \mathrm{cov}\{\mu_1,\theta_1\mu_1+(1-\theta_1)\mu_2\},\\
		K_6\ &:=\ \ \mathrm{cov}\{\mu,\nu\}.
	\end{align*}
	Now, using Lemma \ref{DC1-omega-Lemma} on $K_i,\ i=1,2,3,4,5,6$ and $\Lambda$, we get $\{x\in X\mid x \text{ satisfies case } (i)\}\cap NR(f)\cap I_{\varphi}(f)$ is strongly distributional chaotic for any $i\in\{1,2,3,4,5,6\}$ by Lemma \ref{G_K-subset-omega} and Proposition \ref{proposition of distal}. Using Lemma \ref{DC1-omega-Lemma} on $K_i,\ i=1,2,3,4,5,6$ and $\Theta$, we get $\{x\in X\mid x \text{ satisfies case } (i)\}\cap NR(f)\cap I_{\varphi}(f)$ is strongly distributional chaotic for any $i\in\{1',2',3',4',5',6'\}$ by Lemma \ref{G_K-subset-omega} and Proposition \ref{proposition of distal}. 
\end{proof}

\subsection{Proof of Theorem \ref{thm-Level}}
Now we state a abstract result on strongly distributional chaos of $\{x\in X\mid x \text{ satisfies case } (i)\}\cap NR(f)\cap R_\varphi(a).$ Combining with Theorem \ref{thm-2} and Theorem \ref{thm-3}, we have Theorem \ref{thm-Level}.
\begin{Thm}\label{thm-It}
	Suppose $(X,  f)$ is positively expansive, mixing and has the exponential shadowing property.  Let $\varphi$ be a continuous function on $X$. If $Int(L_{\varphi})\neq\emptyset$, then for any $a\in Int(L_{\varphi})$,   $\{x\in X\mid x \text{ satisfies case } (i)\}\cap NR(f)\cap R_\varphi(a)$ is strongly distributional chaotic for any $i\in\{1,2,3,4,5,6,1',2',3',4',5',6'\}.$
\end{Thm}
\begin{proof}
	Take $\Lambda\subsetneq \Theta \subsetneq X$ satisfying Lemma \ref{lambda-theta-level}. By Proposition \ref{prop-IrregularMultianalysis-22222}, we can take $\mu_1,\mu_2,\mu_3,\mu_4\in  \mathcal M_f^e(\Lambda)$ with mutually disjoint support satisfying
	$$\int\varphi d\mu_1<\int\varphi d\mu_2<a<\int\varphi d\mu_3<\int\varphi d\mu_4.$$
	Similarly, $\bigcup_{i=1}^{4}S_{\mu_i}\neq \Lambda$ by Proposition \ref{prop-IrregularMultianalysis-22222}. Now, we can choose proper $\{\theta_i\}_{i=1}^{3}\subseteq(0,1)$ such that and
	$$\theta_1\int\varphi d\mu_{1}+(1-\theta_1)\int\varphi d\mu_{3}=a;$$
	$$\theta_2\int\varphi d\mu_{2}+(1-\theta_2)\int\varphi d\mu_{4}=a;$$
	$$\theta_3\int\varphi d\mu_{1}+(1-\theta_3)\int\varphi d\mu_{4}=a.$$
	Denote
	$$\nu_1=\theta_1\mu_{1}+(1-\theta_1)\mu_{3};$$
	$$\nu_2=\theta_2\mu_{2}+(1-\theta_2)\mu_{4};$$
	$$\nu_3=\theta_3\mu_{1}+(1-\theta_3)\mu_{4}.$$
	By Lemma \ref{lemma-full} and Lemma \ref{lemma-ergo}, there are $\omega_1,\omega_2\in\mathcal M_f^e(X)$ with full support and 
	$$\int\varphi d\omega_1<a<\int\varphi d\omega_2.$$ Now, we can choose proper $\theta\subseteq(0,1)$ such that $$\theta\int\varphi d\omega_{1}+(1-\theta)\int\varphi d\omega_{2}=a.$$
	Denote $\mu=\theta\omega_1+(1-\theta)\omega_2.$
	Let
	\begin{align*}
		K_1\ &:=\ \ \{\nu_1\},\\
		K_2\ &:=\ \ \mathrm{cov}\{\nu_1,\mu\}\cup\mathrm{cov}\{\nu_2,\mu\},\\
		K_3\ &:=\ \ \mathrm{cov}\{\nu_1,\mu\},\\
		K_4\ &:=\ \ \mathrm{cov}\{\nu_1,\nu_2\},\\
		K_5\ &:=\ \ \mathrm{cov}\{\nu_1,\nu_3\},\\
		K_5\ &:=\ \ \{\mu\}.
	\end{align*}
    Now, using Lemma \ref{DC1-omega-Lemma} on $K_i,\ i=1,2,3,4,5,6$ and $\Lambda$, we get $\{x\in X\mid x \text{ satisfies case } (i)\}\cap NR(f)\cap R_\varphi(a)$ is strongly distributional chaotic for any $i\in\{1,2,3,4,5,6\}$ by Lemma \ref{G_K-subset-omega} and Proposition \ref{proposition of distal}. Using Lemma \ref{DC1-omega-Lemma} on $K_i,\ i=1,2,3,4,5,6$ and $\Theta$, we get $\{x\in X\mid x \text{ satisfies case } (i)\}\cap NR(f)\cap R_\varphi(a)$ is strongly distributional chaotic for any $i\in\{1',2',3',4',5',6'\}$ by Lemma \ref{G_K-subset-omega} and Proposition \ref{proposition of distal}.  
\end{proof}

\subsection{Recurrent Points That Are Not Transitive}\label{recnotdense}
Recall that $$ND=NR\sqcup(Rec\setminus Tran).$$ In previous sections, we have obtained that for mixing expanding maps or  transitive Anosov diffeomorphisms, the set of non-recurrent set $NR(f)$ is strongly distributional chaotic. 
In this subsection we show that $Rec\setminus Tran$ is distributional chaotic of type 1.  
\begin{Thm}\label{thm-lv}
	Suppose $(X,  f)$ is positively expansive, mixing and has the shadowing property.  Let $\varphi$ be a continuous function on $X$. If $I_{\varphi}(f)\neq\emptyset$, then  for any $i\in\{1,2,3,4,5,6\}$
	\begin{description}
		\item[(I)] $\{x\in X\mid x \text{ satisfies case } (i)\}\cap (Rec(f)\setminus Tran(f))\cap I_{\varphi}(f)$ contains an uncountable DC1-scrambled set,
		\item[(II)] $\{x\in X\mid x \text{ satisfies case } (i)\}\cap (Rec(f)\setminus Tran(f))\cap R_\varphi(a)$ contains an uncountable DC1-scrambled set for any $a\in Int(L_{\varphi}).$
	\end{description}
\end{Thm}
\begin{proof}
	We first recall the following result.
	\begin{Lem}\cite[Theorem F]{CT}\label{lemma-A}
		Suppose that $(X, f)$ is a dynamical system with the specification property and let $K$ be a connected non-empty compact subset of $\mathcal{M}_{f}(X)$. If there is a $\mu \in K$ such that $\mu=\theta \mu_{1}+(1-\theta) \mu_{2}\left(\mu_{1}=\mu_{2}\right.$ could happen $)$ where $\theta \in[0,1]$, and $G_{\mu_{1}}$, $G_{\mu_{2}}$ both have distal a pair, then there exists an uncountable DC1-scrambled set $S_{K} \subseteq G_{K} \cap Tran(f).$
	\end{Lem}
    Note that for any invariant compact subset $\Lambda\subsetneq X,$ we have $Tran(f|_\Lambda)\subseteq Rec(f)\setminus Tran(f).$
    Now, using Lemma \ref{lemma-A} on $K_i$ and $\Lambda$ of Theorem \ref{thm-Is}, we get $\{x\in X\mid x \text{ satisfies case } (i)\}\cap (Rec(f)\setminus Tran(f))\cap I_{\varphi}(f)$ contains an uncountable DC1-scrambled set for any $i\in\{1,2,3,4,5,6\}$ by Lemma \ref{G_K-subset-omega} and Proposition \ref{proposition of distal}.  Using Lemma \ref{lemma-A} on $K_i$ and $\Lambda$ of Theorem \ref{thm-It}, we get $\{x\in X\mid x \text{ satisfies case } (i)\}\cap (Rec(f)\setminus Tran(f))\cap R_\varphi(a)$ contains an uncountable DC1-scrambled set for any $i\in\{1,2,3,4,5,6\}$ by Lemma \ref{G_K-subset-omega} and Proposition \ref{proposition of distal}. 
\end{proof}

\section{Other Dynamical Systems}\label{section-application}
In this section, we apply the results in the previous sections to more systems, including  mixing subshifts of finite type, $\beta$-shifts and nonuniformly hyperbolic systems. Before that, we give a abstract result as a generalization of Theorem \ref{thm-Is}, \ref{thm-It} and \ref{thm-lv}.
\begin{Thm}\label{thm-Iu}
	Suppose that $(X,f)$ is a dynamical system, and $Y$ is an invariant compact subset of $X.$  Assume that $(Y,f|_Y)$ is  positively expansive, mixing and has the exponential shadowing property.
	Let $\varphi$ be a continuous function on $X$. If $I_{\varphi}(f)\cap Y\neq\emptyset$, then   for any $i\in\{1,2,3,4,5,6,1',2',3',4',5',6'\}$
	\begin{description}
		\item[(I)] $\{x\in X\mid x \text{ satisfies case } (i)\}\cap NR(f)\cap I_{\varphi}(f)$ is strongly distributional chaotic,
		\item[(II)] $\{x\in X\mid x \text{ satisfies case } (i)\}\cap NR(f)\cap R_\varphi(a)$ is strongly distributional chaotic for any $a\in Int(L_{\varphi}^Y).$
	\end{description}
    And for any $i\in\{1,2,3,4,5,6\}$
    \begin{description}
    	\item[(III)] $\{x\in X\mid x \text{ satisfies case } (i)\}\cap (Rec(f)\setminus Tran(f))\cap I_{\varphi}(f)$ contains an uncountable DC1-scrambled set,
    	\item[(IV)] $\{x\in X\mid x \text{ satisfies case } (i)\}\cap (Rec(f)\setminus Tran(f))\cap R_\varphi(a)$ contains an uncountable DC1-scrambled set  for any $a\in Int(L_{\varphi}^Y).$
    \end{description}
\end{Thm}
\begin{proof}
	Note that $NR(f|_Y)\subseteq NR(f),$ $I_{\varphi}(f|_Y)\subseteq I_{\varphi}(f),$ and $R_\varphi(a)\cap Y\subseteq R_\varphi(a).$ Then by Theorem \ref{thm-Is}, \ref{thm-It} and \ref{thm-lv} we finish the proof.
\end{proof}

\subsection{Symbolic Dynamics}
\subsubsection{Subshifts of Finite Type}
For any finite alphabet $A$, the full symbolic space is the set $A^{\mathbb{Z}}=\{\cdots x_{-1}x_{0}x_{1}\cdots : x_{i}\in A\}$, which is viewed as a compact topological space with the discrete product topology. The set $A^{\mathbb{N^{+}}}=\{x_{1}x_{2}\cdots : x_{i}\in A\}$ is called one side full symbolic space. The shift action on one side full symbolic space is defined by
$$\sigma:\ A^{\mathbb{N^{+}}}\rightarrow A^{\mathbb{N^{+}}},\ \ \ x_{1}x_{2}\cdots\mapsto x_{2}x_{3}\cdots.$$
$(A^{\mathbb{N^{+}}},\sigma)$ forms a dynamical system under the discrete product topology which we called a shift. We equip $A^{\mathbb{N^{+}}}$ with a compatible metric $d$ given by
\begin{equation}
	d(\omega,\gamma)=
	\begin{cases}
		n^{-\min\{k\in\mathbb{N^{+}}:\omega_{k}\neq \gamma_{k}\}},&\omega\neq \gamma,\\
		0,&\omega= \gamma.
	\end{cases}
\end{equation}
where $n=\#A$. A closed subset $X\subseteq A^{\mathbb{N^{+}}}$ or $A^{\mathbb{Z}}$ is called subshift if it is invariant under the shift action $\sigma$. $w\in A^{n}\triangleq \{x_1x_2\cdots x_n:\ x_{i}\in A\}$ is a word of subshift $X$ if there is an $x\in X$ and $k\in \mathbb{N}^{+}$ such that $w=x_kx_{k+1}\cdots x_{k+n-1}$. Here we call $n$ the length of $w$, denoted by $|w|$. The language of a subshift $X$, denoted by $\mathcal L(X)$, is the set of all words of $X$. Denote $\mathcal L_{n}(X)\triangleq \mathcal L(X)\bigcap A^{n}$ all the words of $X$ with length $n$.

Let $B$ be a set of words occurring in $A^{\mathbb{N}^+}$. Then $$\Lambda_{B}:=\{x\in A^{\mathbb{N}^+}\mid\text{no block} \ w\in B\ \text{occurs in}\ x\}$$ is a compact $\sigma$-invariant subset. $(\Lambda_{B},\sigma)$ is a subshift defined by $B$. A subshift $(\Lambda,\sigma)$ is called a (one-side) subshift of finite type if there exists a finite word set $B$ such that $\Lambda_{B}=\Lambda$.
By \cite{Walters2}, a subshift satisfies shadowing property if and only if it is a subshift of finite type. By \cite[Proposition 6.4]{HT} every subshift with shadowing property satisfies exponential shadowing property.
So every  subshift of finite type has exponential shadowing property.
As a subsystem of full shift, it is positively expansive. So we have the following.
\begin{Thm}\label{thm-5}
	The results of Theorem \ref{thm-Is} and \ref{thm-It} hold for every mixing (one-side) subshift of finite type.
\end{Thm}
\begin{Rem}
	Theorem \ref{thm-5} is also true for every mixing (two-side) subshift of finite type.
\end{Rem}
\subsubsection{$\beta$-Shifts}
{ Next we present one type of subshift, (one-side) $\beta$-shift, basically referring to \cite{R,Sm,PS}.} 
Let $\beta > 1$ be a real number. We denote by $[x]$ and $\{x\}$ the integer and fractional part of the real number $x$.
Considering the $\beta$-transformation $f_{\beta}:[0,1)\rightarrow [0,1)$ given by $$f_{\beta}(x)=\beta x\ (\mathrm{mod}\ 1).$$
For $\beta \notin \mathbb{N}$, let $b=[\beta]$ and for $\beta \in \mathbb{N}$, let $b=\beta-1$. Then we split the interval $[0,1)$ into $b+1$ partition as below

$$J_0=\left[0,\frac{1}{\beta}\right),\ J_1=\left[\frac{1}{\beta},\frac{2}{\beta}\right), \cdots,\ J_b=\left[\frac{b}{\beta},1\right).$$
For $x\in[0,1)$, let $i(x,\beta)=(i_n(x,\beta))_1^{\infty}$ be the sequence given by $i_n(x,\beta)=j$ when $f_{\beta}^{n-1}(x)\in J_j$. We call $i(x,\beta)$ the greedy $\beta$-expansion of $x$ and we have
$x=\sum_{n=1}^{\infty}i_n(x,\beta)\beta^{-n}.$
We call $(\Sigma_{\beta},\sigma)$ (one-side) $\beta$-shift, where $\sigma$ is the shift map, $\Sigma_{\beta}$ is the closure of $\{i(x,\beta)\}_{x\in [0,1)}$ in $\prod_{i=1}^{\infty}\{0,1,\cdots,b\}$.

From the discussion above, we can define the greedy $\beta$-expansion of 1, denoted by $i(1,\beta)$. Parry showed that the set of sequence with belong to $\Sigma_{\beta}$ can be characterised as
$$\omega \in \Sigma_{\beta} \Leftrightarrow f^k(\omega) \leq i(1,\beta)\ \mathrm{for\ all}\ k\geq 1,$$
where $\leq$ is taken in the lexicographic ordering \cite{P}. By the definition of $\Sigma_{\beta}$ above, $\Sigma_{\beta_1}\subsetneq\Sigma_{\beta_2}$ for $\beta_1<\beta_2$(\cite{P}). 

From \cite{Sm}, $\{\beta\in (1,+\infty):(\Sigma_{\beta},\sigma)\text{ has the shadowing  property} \}$ is dense in $(1,+\infty)$. For every $\beta_1>1$, there exists $1<\beta_2<\beta_1$ such that and $(\Sigma_{\beta_{2}},\sigma)$ satisfies the shadowing property. By \cite[Proposition 6.4]{HT} every subshift with shadowing property satisfies exponential shadowing property.
So $(\Sigma_{\beta_{2}},\sigma)$ has exponential shadowing property. By definition every $\beta$-shift is positively expansive and mixing. So we have the following.
\begin{Thm}\label{thm-6}
	The results of Theorem \ref{thm-Iu} hold for every (one-side) $\beta$-shift.
\end{Thm}
\begin{Rem}
	Theorem \ref{thm-6} is also true for every (two-side) $\beta$-shift.
\end{Rem}

\subsection{Homoclinic Classes and Hyperbolic Ergodic Measures}
\subsubsection{Homoclinic Classes}
Let $M$ be  a compact Riemannian manifold $M,$ and $\operatorname{Diff}^{1}(M)$ denote the space of $C^{1}$ diffeomorphisms on $M.$ Given $f\in \operatorname{Diff}^{1}(M),$ we recall that the homoclinic class of a hyperbolic saddle $p$, denoted by $H(p),$ is the closure of the set of hyperbolic saddles $q$ homoclinically related to $p$ (the stable manifold of the orbit of $q$ transversely meets the unstable one of the orbit of $p$ and vice versa). In this subsection we consider homoclinic class $H(p)$ in which there is a  hyperbolic periodic point $q$ such that $q$ is homoclinic related to $p$ and the periods of $p$ and $q$ are coprime.  Since $p_1$ is homoclinic related to $p,$ there is $\Lambda\subseteq H(p)$ such that $(\Lambda,f)$ a transitive locally maximal hyperbolic set  that contains $p$ and $q$. By Theorem \ref{thm-2}, $(\Lambda,f)$ is expansive and has the exponential shadowing property. Next we show that it is mixing.
\begin{Prop}\label{proposition-primeperiodic}
	Suppose that a dynamical system $(X,f)$ is transitive and has the shadowing property. If there exist two periodic points $p_{1}$ and $p_{2}$ shch that $n_{1}$ and $n_{2}$ are coprime where $n_{i}$ is the period of $p_{i}$ for $i=1,2$, then $(X,f)$ is mixing.
\end{Prop}
\begin{proof}
	Since $n_{1}$ and $n_{2}$ are coprime, there exist two integers $m_{1}$ and $m_{2}$ such that $m_{1}n_{1}+m_{2}n_{2}=1$. We assume that $m_{2}>0$, then for any $n\geq |m_{1}|n_{1}^{2}$, $n$ is expressed as $n=mn_{1}+l$ where $m\geq |m_{1}|n_{1}$ and $0\leq l<n_{1}$ are integers. Then one has 
	\begin{equation*}
		n=mn_{1}+l=mn_{1}+l(m_{1}n_{1}+m_{2}n_{2})=(m+lm_{1})n_{1}+lm_{2}n_{2}
	\end{equation*}
	and $m+lm_{1}\geq |m_{1}|n_{1}+lm_{1}>0$. So for any $n\geq |m_{1}|n_{1}^{2}$ there exist nonnegative $l_{1}$ and $l_{2}$ such that $n=l_{1}n_{1}+l_{2}n_{2}$.
	
	For any nonempty open sets $U$ and $V$ in $X$, there exist $x_{1}$, $x_{2}$ and $\varepsilon>0$ such that $B(x_{1},2\varepsilon)\subset U$ and $B(x_{2},2\varepsilon)\subset V$. Let $\delta>0$ be provided for the $\varepsilon$ by shadowing property. Let $0<\delta_{0}<\min\{\delta,\varepsilon\}$. Since $(X,f)$ is transitive, there exist $y_{1}\in B(x_{1},\delta_{0})$, $y_{2}\in B(p_{1},\delta_{0})$, $y_{3}\in B(p_{2},\delta_{0})$ and $s_{1}$, $s_{2}$, $s_{3}\in\mathbb{N^{+}}$ such that $f^{s_{1}}(y_{1})\in B(p_{1},\delta_{0})$, $f^{s_{2}}(y_{2})\in B(p_{2},\delta_{0})$ and $f^{s_{3}}(y_{3})\in B(x_{2},\delta_{0})$. For any $n\geq |m_{1}|n_{1}^{2}$ there exist nonnegative $l_{1}$ and $l_{2}$ such that $n=l_{1}n_{1}+l_{2}n_{2}$. For the $n$ we define $\mathfrak{C}$ by
	$$\mathfrak{C}^n=orb(y_1,s_1)orb(p_1,l_1n_1)orb(y_2,s_2)orb(p_2,l_2n_2)orb(y_3,s_3+1).$$
	Then $\mathfrak{C}^n$ is a $\delta$-chain  connecting $x_1$ and $x_2$. By shadowing property there exists $z\in X$ such that $\mathfrak{C}^n$ is $\varepsilon$-traced by $z$. Then we have $z\in B(x_{1},2\varepsilon)\subset U$ and $f^{s_{1}+l_{1}n_{1}+s_{2}+l_{2}n_{2}+s_{3}}(z)\in B(x_{2},2\varepsilon)\subset V$. So $U\cap f^{-t}V\neq \emptyset$ for any $t\geq s_{1}+s_{2}+s_{3}+|m_{1}|n_{1}^{2}$, i.e., $(X,f)$ is mixing.
\end{proof}
Thus $(\Lambda,f)$ is expansive, mixing and has the exponential shadowing property. So we have the following.
\begin{Thm}\label{thm-4}
	For every homoclinic class $H(p)$ in which there is a  hyperbolic periodic point $q$ such that $q$ is homoclinic related to $p$ and the periods of $p$ and $q$ are coprime, the results of Theorem \ref{thm-Iu} hold.
\end{Thm}

\subsubsection{Hyperbolic Ergodic Measures}
Given $f\in \operatorname{Diff}^{1}(M),$ we recall that an ergodic $f$-invariant Borel probability measure is said to be hyperbolic if it has positive and negative but no zero Lyapunov exponents, 
In \cite{Katok}, for any ergodic and non-atomic hyperbolic measure $\mu$ of a $C^{1+\alpha}$ diffeomorphism the author finds a hyperbolic periodic point $p$ such that $S_\mu$ is contained in the closure of the transverse homoclinic points of $p$, i.e., $H(p)$. Following Katok's idea, in \cite{LSVW} for any weakly mixing hyperbolic measure $\mu$ of a $C^{1+\alpha}$ diffeomorphism G. Liao \textit{et al} find a hyperbolic periodic point $p$ and a mixing locally maximal hyperbolic set $\Lambda$ such that $S_\mu\subset H(p)$ and $p\in \Lambda \subset H(p)$.  Thus we can state a similar result as a corollary of Theorem \ref{thm-4} for hyperbolic weakly mixing measures. 
\begin{Cor}
	Let $f:M\rightarrow M$ be a  diffeomorphism on a compact Riemannian manifold. Assume that $f$ is $C^{1+\alpha}$ and there is a weakly mixing hyperbolic measure. Then    the results of Theorem \ref{thm-Iu} hold.
\end{Cor}

\section*{Acknowlegements}   
The authors are
supported by the National Natural Science Foundation of China (grant No. 12071082, 11790273) and in part by Shanghai Science and Technology Research Program (grant No. 21JC1400700).

\end{document}